\documentclass[11pt]{article}

\usepackage[margin=1in]{geometry}
\usepackage{amsmath, amsfonts, amsthm, amssymb}
\usepackage{mathrsfs}
\usepackage{hyperref}
\usepackage{tikz}
\usetikzlibrary{arrows.meta,calc,patterns}
\usepackage{booktabs}   

\hypersetup{
  colorlinks=true,
  linkcolor=blue,
  citecolor=blue,
  urlcolor=blue,
}

\newcommand{\NP}{\Gamma_+}

\DeclareMathOperator{\Exc}{Exc}

\newtheorem{theorem}{Theorem}[section]
\newtheorem{proposition}[theorem]{Proposition}
\newtheorem{corollary}[theorem]{Corollary}
\newtheorem{lemma}[theorem]{Lemma}
\theoremstyle{definition}
\newtheorem{definition}[theorem]{Definition}
\newtheorem{remark}[theorem]{Remark}
\newtheorem{example}[theorem]{Example}
\newtheorem{hypothesis}[theorem]{Hypothesis}

\title{Geometric Criteria for Morrey Admissibility via the Real Log-Canonical Threshold}

\author{
Nivaldo Grulha\thanks{Instituto de Ci\^encias Matem\'aticas e de Computa\c{c}\~ao (ICMC), Universidade de S\~ao Paulo, S\~ao Carlos, Brazil.}
\and
Andr\'ea Prokopczyk\thanks{Universidade Estadual Paulista (UNESP), S\~ao Jos\'e do Rio Preto, Brazil.}
}

\date{}

\begin{document}

\maketitle

\begin{abstract}
This paper studies the local integrability of the gradient of singular
interaction kernels in aggregation equations, with a view toward the
well-posedness theory in Morrey spaces of
\cite{suleiman2020aggregation,suleiman2023existence}. The velocity field is
a nonlocal convolution involving $\nabla K$. Near the singularity we take
$\nabla K$ to be analytically comparable, in the two-sided sense, to
$|\nabla f|\,|f|^{-(\kappa+1)}$ for some $\kappa>0$ and some real-analytic
$f$ vanishing at the origin; the model case is $K=|f|^{-\kappa}$.

The exact local integrability threshold of $\nabla K$ admits a complete
geometric description under this assumption:
\[
\nabla K \in L^p_{\mathrm{loc}}
\quad\Longleftrightarrow\quad
p < p^*(\kappa):=\min_i \frac{a_i+1}{(\kappa+1)\nu_i - M_i},
\]
where $\nu_i$, $M_i$, $a_i$ are the vanishing orders of $f$, the vanishing
orders of its Jacobian ideal, and the discrepancy exponents of a real
log-resolution of $f$ together with its Jacobian ideal. The real
log-canonical threshold of $f$ gives a computable lower bound
$\overline p(\kappa):=\mathrm{rlct}_0(f)/(\kappa+1)$ for $p^*(\kappa)$.
Equality holds exactly when some divisor computing $\mathrm{rlct}_0(f)$ has
vanishing Jacobian order $M_i=0$---necessarily a reduced, generically
smooth component of the strict transform---so the bound is sharp precisely
in this mildly singular regime, and strict wherever the threshold is
instead governed by a divisor over a critical point of $f$, in particular
for every isolated zero in dimension at least two. Quasi-homogeneous
singularities admit the closed formula $|w|/(\kappa d + \max_i w_i)$ in
terms of the weights $w$ and weighted degree $d$; the classical Newtonian
kernel is recovered exactly within this framework.

This geometric threshold connects to well-posedness through
\cite[Thm.~3.1]{suleiman2023existence}, whose only hypothesis on the
kernel---as stated there---is the global condition
$\nabla K\in L^1(\mathbb R^n)$, independent of the Morrey scaling parameter
$\lambda$ of the solution space. That condition follows from the local
criterion $p^*(\kappa)>1$ together with a mild global truncation of $K$
away from its singularity, with $\kappa+1<\mathrm{rlct}_0(f)$ a computable
sufficient condition. For the Newtonian kernel the RLCT bound is exactly
borderline, yet the exact threshold certifies the classical theory for
every dimension $d\ge3$ and, once certified, simultaneously for the
\emph{entire} admissible range of $\lambda$---a case where the divisorial
threshold does work the RLCT alone cannot.

The proof reduces, through resolution of singularities, to a monomial
model that yields the integrability estimates directly. The resulting
criterion covers isotropic kernels of Riesz type together with anisotropic
kernels modeled on analytic divisors, and the threshold governing
well-posedness turns out to be resolution-theoretic (birational) data
attached to the singularity.
\end{abstract}



\section*{Introduction}
Nonlocal aggregation equations model chemotaxis, swarming, and related
collective phenomena, where long-range interactions are represented by
convolution terms with singular kernels. A basic example is
\begin{equation}\label{eq:agr}
\partial_t u = \Delta u - \nabla\cdot\bigl(u\,(\nabla K * u)\bigr),
\end{equation}
where \(u\) is the density and \(K\) an interaction potential. Solutions of
\eqref{eq:agr} depend on the singularity of \(K\), and specifically on the
local integrability of \(\nabla K\).

Aggregation equations of this type trace back to classical models of
chemotaxis and biological swarming. Their analysis has since developed
across several functional settings: $L^p$ spaces \cite{bertozzi2012lp},
Sobolev and measure-valued frameworks
\cite{carrillo2011global,bedrossian2015existence}, and continuum
descriptions of collective dynamics
\cite{carrillo2014derivation,topaz2004swarming,mogilner1999nonlocal}. Across
this literature, the nature of the singularity of the interaction potential
governs both the qualitative behavior of solutions and the onset of
blow-up.

Morrey spaces are a natural setting once diffusion is present, capturing
the interplay between scaling, concentration, and nonlocal effects.
Suleiman, Precioso, and Prokopczyk
\cite{suleiman2020aggregation,suleiman2023existence} used this framework to
establish global well-posedness for \eqref{eq:agr} in critical Morrey
spaces, under conditions ensuring the boundedness of the bilinear map
\[
(u,v)\longmapsto u\,(\nabla K * v)
\]
that are formulated directly in terms of local integrability and
convolution estimates for $\nabla K$. Morrey spaces have a long history in
parabolic and fluid equations, beginning with Kozono--Yamazaki and Taylor
\cite{kozono1994semilinear,taylor1992analysis} and developed within Kato's
theory of critical spaces \cite{kato1992strong}; see also
\cite{lemarierieusset2016navier,miao2011wellposedness,giga1999cauchy}.

What \cite{suleiman2020aggregation,suleiman2023existence} leaves open is
which geometric feature of the kernel fixes the precise admissibility
regime. For isotropic kernels $K(x)\sim|x|^{-\alpha}$, scaling arguments
answer this directly. Anisotropic or degenerate singularities---arising
naturally in cross-diffusion and multi-species models
\cite{carrillo2020zoology}---call for a more refined analysis.

Here we show that divisorial data from a real resolution of singularities
govern the local integrability of $\nabla K$, with the \emph{real log
canonical threshold} (RLCT) supplying a computable leading-order
approximation. Assuming
$|\nabla K| \asymp |\nabla f|/|f|^{\kappa+1}$ for some $\kappa>0$ near a
zero of a real-analytic function $f$ (model case $K=|f|^{-\kappa}$), we
obtain a complete geometric characterisation of the exact local
integrability threshold:
$\nabla K \in L^p_{\mathrm{loc}}$ if and only if
$p < p^*(\kappa):=\min_i (a_i+1)/\bigl((\kappa+1)\nu_i - M_i\bigr)$, the
minimum running over the real divisorial components of a log-resolution of
$f$ together with its Jacobian ideal; $\nu_i$ and $M_i$ are the
corresponding vanishing orders of $f$ and of the Jacobian ideal, and $a_i$
the discrepancies. The RLCT gives the computable lower bound
$\mathrm{rlct}_0(f)/(\kappa+1)$ for this threshold.

As stated in \cite[Thm.~3.1]{suleiman2023existence}, the only requirement
on the kernel is the \emph{global} condition $\nabla K\in
L^1(\mathbb R^n)$---not a local $L^p_{\mathrm{loc}}$ condition tied to the
Morrey scaling parameter $\lambda=n-p$ of the solution space. Here
$p^*(\kappa)>1$ is exactly the local condition for the singular part of the
kernel to lie in $L^1_{\mathrm{loc}}$; together with a mild global
truncation of $K$ away from its singularity (made precise in
Corollary~\ref{cor:morrey_bridge} via the condition
$(1-\chi)\nabla K\in L^1$ on the truncated far field), it yields the full
global hypothesis $\nabla K\in L^1(\mathbb R^n)$ required by
\cite[Thm.~3.1]{suleiman2023existence}, with $\kappa+1<\mathrm{rlct}_0(f)$
a computable sufficient condition; see Section~\ref{sec:morrey_bridge}. The
passage from this global $L^1$ hypothesis to the full bilinear estimate
above, and hence to well-posedness of the aggregation equation, is
established in \cite[Thm.~3.1]{suleiman2023existence} under exactly this
hypothesis; we do not reprove those bilinear estimates here.

Resolution of singularities produces the vanishing orders and discrepancies
that control the $L^p$-integrability of $\nabla K$. The RLCT captures the
leading contribution; the finer invariants $M_i$, measuring the vanishing
of the gradient along each divisor, fix the sharp threshold. On any divisor
$E_i$ lying over a critical point of $f$, $M_i>0$, so
$(\kappa+1)\nu_i-M_i<(\kappa+1)\nu_i$ and the ratio
$(a_i+1)/\bigl((\kappa+1)\nu_i-M_i\bigr)$ strictly exceeds the RLCT-based
term $(a_i+1)/\bigl((\kappa+1)\nu_i\bigr)$: the gap between $p^*(\kappa)$
and $\overline p(\kappa)$ is produced specifically by $M_i>0$, which occurs
in particular for every isolated zero in dimension at least two. The sharp
threshold is thus genuinely divisorial rather than an RLCT phenomenon.

The classical isotropic theory sits inside this picture as a special case:
the Newtonian kernel is recovered \emph{exactly} as the instance $f=|x|^2$,
$\kappa=(d-2)/2$ of the threshold formula. Related mechanisms appear in the
study of integrability of reciprocal functions
\cite{mustata2002singularities}, its real-analytic and combinatorial
refinement for hyperplane arrangements \cite{kostawindisch2026}, and the
analysis of oscillatory integrals
\cite{varchenko1976newton,collins2017real}.

Hypersurface zero sets are treated on the same footing as isolated zeros:
components of the strict transform enter the threshold alongside
exceptional divisors. Singular sets requiring genuinely stratified or
valuation-theoretic refinements---kernels whose singular behavior is not
governed by a single analytic divisor---are left for future work.

\medskip
\noindent\textbf{Structure of the paper.}
Section~\ref{sec:expanded_background} reviews the Morrey space framework of
\cite{suleiman2020aggregation,suleiman2023existence}.
Section~\ref{sec:rlct_morrey} develops the resolution-based integrability
criteria above and establishes, in Section~\ref{sec:morrey_bridge}, the
precise bridge to \cite[Thm.~3.1]{suleiman2023existence}.
Section~\ref{sec:examples} works through examples based on Newton
polyhedra, including a discussion of when the geometric threshold secures
the global $L^1$ hypothesis (Regime A of Remark~\ref{rem:three_regimes}).

\section{Analytic framework for aggregation equations in Morrey spaces}
\label{sec:morrey_framework}
\label{sec:expanded_background}

This section establishes the analytical framework used throughout the paper
to construct mild solutions of the aggregation equation.
The presentation is a synthesis of the approaches developed in
\cite{suleiman2020aggregation,suleiman2023existence}.

The Morrey space framework for parabolic equations was systematically
developed by Kozono--Yamazaki and Taylor
\cite{kozono1994semilinear,taylor1992analysis}, building on Kato's
pioneering work \cite{kato1992strong}. The scaling-invariant approach employed
here is part of a broader program treating critical spaces for evolution
equations; see \cite{lemarierieusset2016navier} for a comprehensive modern
treatment and \cite{miao2011wellposedness,giga1999cauchy} for related
techniques in critical function spaces.

\medskip
The framework relies on three ingredients:
\begin{enumerate}
\item the Duhamel formulation of the equation;
\item structural and scaling properties of Morrey-type spaces;
\item semigroup and bilinear estimates leading to a fixed-point argument.
\end{enumerate}
\subsection{Aggregation equation and mild formulation}

We consider the viscous aggregation equation
\begin{equation}\label{eq:agg}
u_t=\Delta u-\nabla\!\cdot\!\big(u(\nabla K*u)\big),
\qquad x\in\mathbb R^n,\ t>0,
\end{equation}
with initial data
\[
u(x,0)=u_0(x).
\]

The linearized problem is the heat equation
\[
u_t=\Delta u, \qquad u(0)=u_0,
\]
whose solution is given by the heat semigroup
\[
G(t)u_0=g(\cdot,t)*u_0,
\qquad
g(x,t)=(4\pi t)^{-n/2}e^{-|x|^2/(4t)}.
\]

By Duhamel's principle, a mild solution of \eqref{eq:agg}
satisfies the integral equation
\begin{equation}\label{eq:mild}
u(t)=G(t)u_0+B(u,u)(t),
\end{equation}
where the bilinear operator is
\begin{equation}\label{eq:bilinear_operator}
B(u,v)(t)
:=-\int_0^t
\nabla_x G(t-s)\big(u(\nabla K * v)\big)(s)\,ds .
\end{equation}

The negative sign in \eqref{eq:bilinear_operator} corresponds exactly to the
divergence structure of the nonlinear transport term.

\subsection{Morrey spaces}

For \(1\le p<\infty\) and \(0\le\lambda<n\),
\[
M_{p,\lambda}(\mathbb R^n)
=\{f\in L^p_{\rm loc}:\|f\|_{M_{p,\lambda}}<\infty\},
\]
with norm
\[
\|f\|_{M_{p,\lambda}}
=\sup_{x_0\in\mathbb R^n,R>0}
R^{-\lambda/p}\|f\|_{L^p(B_R(x_0))}.
\]

Important special cases:
\[
M_{p,0}=L^p,
\qquad
M_{1,0}=\mathcal M(\mathbb R^n).
\]

\paragraph{Scaling property}
\[
\|f(\alpha\cdot)\|_{M_{p,\lambda}}
=\alpha^{-\frac{n-\lambda}{p}}\|f\|_{M_{p,\lambda}}.
\]

This scaling is compatible with the invariance of the aggregation equation.

\subsection{Functional inequalities in Morrey spaces}

\begin{lemma}[H\"older inequality in Morrey spaces]
\label{lem:holder_morrey}
Let $1 \le p_1,p_2,p_3 < \infty$ and $0 \le \lambda_1,\lambda_2,\lambda_3 < n$
satisfy
\[
\frac1{p_3}=\frac1{p_1}+\frac1{p_2},
\qquad
\frac{\lambda_3}{p_3}
=\frac{\lambda_1}{p_1}+\frac{\lambda_2}{p_2}.
\]
If
\[
f \in M_{p_1,\lambda_1}(\mathbb{R}^n)
\quad \text{and} \quad
g \in M_{p_2,\lambda_2}(\mathbb{R}^n),
\]
then the product $fg$ belongs to $M_{p_3,\lambda_3}(\mathbb{R}^n)$ and
\[
\|fg\|_{M_{p_3,\lambda_3}}
\le
\|f\|_{M_{p_1,\lambda_1}}
\|g\|_{M_{p_2,\lambda_2}} .
\]
\end{lemma}

\begin{lemma}[Young inequalities in Morrey spaces]
\label{lem:young_morrey}
Let $1 \le p < \infty$ and $0 \le \lambda < n$.

\medskip
\noindent
(i) If
\[
g \in L^1(\mathbb{R}^n)
\quad \text{and} \quad
f \in M_{p,\lambda}(\mathbb{R}^n),
\]
then the convolution $g*f$ belongs to $M_{p,\lambda}(\mathbb{R}^n)$ and
\[
\|g*f\|_{M_{p,\lambda}}
\le
\|g\|_{L^1(\mathbb{R}^n)}
\|f\|_{M_{p,\lambda}(\mathbb{R}^n)}.
\]

\medskip
\noindent
(ii) If $\mu$ is a finite Radon measure on $\mathbb{R}^n$ and
\[
f \in M_{p,\lambda}(\mathbb{R}^n),
\]
then the convolution $\mu*f$ belongs to $M_{p,\lambda}(\mathbb{R}^n)$ and
\[
\|\mu*f\|_{M_{p,\lambda}}
\le
|\mu|(\mathbb{R}^n)\,
\|f\|_{M_{p,\lambda}(\mathbb{R}^n)},
\]
where $|\mu|$ denotes the total variation of $\mu$.
\end{lemma}
\subsection{Admissible parameter regime}

Throughout the Morrey framework we assume
\begin{equation}\label{eq:param_cond}
n\ge2,\qquad 0\le\lambda<n-1,
\end{equation}
\[
1<p<q<\infty,
\qquad
p=n-\lambda,
\qquad
\frac1p+\frac1q<1,
\]
and define the time--decay exponent
\[
\eta=\frac12-\frac{n-\lambda}{2q}>0.
\]

\subsection{Scaling--invariant solution space}

We now introduce the functional space in which the fixed--point argument is
performed.

\begin{definition}[Scaling--invariant Morrey solution space]
Let $p,q,\lambda$ and $\eta$ satisfy \eqref{eq:param_cond}.
We define
\[
\mathcal E_q
:=
\Big\{
u:(0,\infty)\to M_{p,\lambda} \;:\;
u\in BC((0,\infty);M_{p,\lambda}),
\quad
t^\eta u(t)\in BC((0,\infty);M_{q,\lambda})
\Big\},
\]
endowed with the norm
\[
\|u\|_{\mathcal E_q}
:=
\sup_{t>0}\|u(t)\|_{M_{p,\lambda}}
+
\sup_{t>0} t^\eta\|u(t)\|_{M_{q,\lambda}}.
\]
\end{definition}

The space $\mathcal E_q$ is invariant under the natural parabolic scaling of the
aggregation equation. Its completeness follows from the completeness of the
Banach spaces $BC((0,\infty);M_{p,\lambda})$ and
$BC((0,\infty);M_{q,\lambda})$; see \cite{suleiman2023existence}.
 Indeed, if $(u_k)$ is a
Cauchy sequence in $\mathcal E_q$, then both
\[
\sup_{t>0}\|u_k(t)-u_\ell(t)\|_{M_{p,\lambda}}
\quad\text{and}\quad
\sup_{t>0}t^\eta\|u_k(t)-u_\ell(t)\|_{M_{q,\lambda}}
\]
converge to zero.

We now recall the analytic ingredients underlying the fixed--point framework
for aggregation equations in Morrey spaces, following closely the
well--posedness theory developed in
\cite{suleiman2020aggregation,suleiman2023existence}.

\begin{lemma}[Abstract fixed point lemma, cf.\ \cite{suleiman2020aggregation,suleiman2023existence}]
\label{lem:fixed_point}

Let \(X\) be a Banach space and \(B:X\times X\to X\) bilinear with
\[
\|B(x_1,x_2)\|_X\le N\|x_1\|_X\|x_2\|_X.
\]
If \(\|y\|_X\le\varepsilon\) and \(4N\varepsilon<1\),
the equation \(x=y+B(x,x)\) has a unique solution.
\end{lemma}

In our setting the Banach space is \(X=\mathcal E_q\), with \(y=G(\cdot)u_0\).

\begin{lemma}[Semigroup estimates]\label{lem:semigroup}
Let \(1\le p_1,p_2<\infty\) and \(0\le\lambda_1,\lambda_2<n\).
If
\[
0\le \frac{n-\lambda_1}{p_1}\le \frac{n-\lambda_2}{p_2},
\]
then for all \(t>0\),
\[
\|G(t)f\|_{M_{p_2,\lambda_2}}
\le C\, t^{-\frac12\left(\frac{n-\lambda_1}{p_1}-\frac{n-\lambda_2}{p_2}\right)}
\|f\|_{M_{p_1,\lambda_1}},
\]
\[
\|\nabla G(t)f\|_{M_{p_2,\lambda_2}}
\le C\, t^{-\frac12\left(1+\frac{n-\lambda_1}{p_1}
-\frac{n-\lambda_2}{p_2}\right)}
\|f\|_{M_{p_1,\lambda_1}}.
\]
\end{lemma}
\begin{lemma}[Bilinear estimates, cf.\ \cite{suleiman2023existence}]
\label{lem:bilinear}
Assume $\nabla K\in L^1(\mathbb R^n)$. Then there exist constants
$C_1,C_2>0$, depending only on $\|\nabla K\|_{L^1(\mathbb R^n)}$ and the
parameters $(p,q,\lambda,\eta)$, such that
\[
\sup_{t>0}\|B(u,v)(t)\|_{M_{p,\lambda}}
\le C_1
\sup_{t>0}t^\eta\|u(t)\|_{M_{q,\lambda}}
\sup_{t>0}\|v(t)\|_{M_{p,\lambda}},
\]
\[
\sup_{t>0}t^\eta\|B(u,v)(t)\|_{M_{q,\lambda}}
\le C_2
\sup_{t>0}t^\eta\|u(t)\|_{M_{q,\lambda}}
\sup_{t>0}t^\eta\|v(t)\|_{M_{q,\lambda}}.
\]
This is \cite[Lemma 4.3]{suleiman2023existence}; the proof combines the
semigroup estimates of Lemma~\ref{lem:semigroup} with the H\"older
inequality of Lemma~\ref{lem:holder_morrey} and the Young inequality
$\|g*f\|_{M_{p,\lambda}}\le\|g\|_{L^1}\|f\|_{M_{p,\lambda}}$ of
Lemma~\ref{lem:young_morrey}(i), applied with $g=\nabla K$; in particular
$C_1,C_2$ are, explicitly, constant multiples of $\|\nabla K\|_{L^1(\mathbb R^n)}$.
We emphasize that, as stated there, this is the \emph{only} hypothesis on
$\nabla K$ used by
\cite{suleiman2020aggregation,suleiman2023existence}: it is a statement
about the \emph{global} $L^1(\mathbb R^n)$ norm of $\nabla K$, not a local
$L^p_{\mathrm{loc}}$ statement for $p\ne1$. This is Mechanism (A) of
Section~\ref{subsec:rlct_definition} below; Mechanism (B) (local
$L^p_{\mathrm{loc}}$ integrability for $p\ne1$) is a variant we record for
completeness but is not what \cite{suleiman2020aggregation,suleiman2023existence}
require.
\end{lemma}

Together, the linear and bilinear bounds allow the contraction lemma to be
applied, yielding global existence and uniqueness of mild solutions to
\eqref{eq:agg} in the space $\mathcal E_q$.

\section{Geometric criteria for Morrey admissibility}
\label{sec:rlct_morrey}

We now turn to the geometric analysis of singular interaction kernels.
Our goal is to identify the mechanism governing the local integrability of
$\nabla K$, which is the key ingredient in the fixed--point framework above.

Rather than relying on isotropic or purely scaling arguments, we make explicit
the geometric structure underlying admissibility thresholds. We show that these
thresholds are organized by resolution--theoretic data associated with the
singularities of the interaction kernel. We begin by recalling the distinction
between ordinary resolutions and log--resolutions, then extract the numerical
data controlling volume distortion under resolution, and finally introduce the
real log--canonical threshold and its relation to Morrey admissibility via
local $L^p$ properties of $\nabla K$.

\subsection*{Notation summary for Section~\ref{sec:rlct_morrey}}

For the convenience of the reader, we collect here the main divisorial symbols
used throughout this section. Throughout, $\kappa>0$ is a fixed exponent and
the model kernel is $K=|f|^{-\kappa}$ (the case $\kappa=1$ corresponds to
$K=1/f$). Let $\pi:\widetilde X\to X$ be a log-resolution of the pair formed
by $f$ and its Jacobian ideal
$\mathcal J_f=(\partial_1 f,\dots,\partial_n f)$, and let $E_i$ denote the
irreducible components of the \emph{total transform} of $(f=0)$ meeting
$\pi^{-1}(0)$; these comprise both the exceptional divisors of $\pi$ and the
components of the strict transform of $\{f=0\}$.

\begin{center}
\small
\renewcommand{\arraystretch}{1.5}
\begin{tabular}{@{}c p{6.7cm} p{3.1cm}@{}}
\toprule
Symbol & Definition & Role \\
\midrule
$\nu_i$ & $\nu_{E_i}(f)$: order of vanishing of $f$ along $E_i$ & controls pole of $K=|f|^{-\kappa}$ \\
$M_i$ & $\nu_{E_i}(\mathcal J_f)$: order of vanishing of $\mathcal J_f$ along $E_i$; $0\le M_i\le\nu_i-1$ & controls pole of $\nabla K$ \\
$a_i$ & discrepancy exponent: $\pi^*(dx)\asymp\prod|y_i|^{a_i}dy$ & Jacobian distortion \\
$\mathrm{rlct}_0(f)$ & $\min_i (a_i+1)/\nu_i$ & integrability of $|f|^{-c}$ \\
$p^*(\kappa)$ & $\min_i (a_i+1)/\bigl((\kappa+1)\nu_i-M_i\bigr)$ & exact gradient-integrability threshold \\
$\overline{p}(\kappa)$ & $\mathrm{rlct}_0(f)/(\kappa+1)$ & computable lower bound for $p^*(\kappa)$ \\
\bottomrule
\end{tabular}
\end{center}

\medskip
On strict-transform components one has $a_i=0$, and on reduced components
along which $\{f=0\}$ is generically smooth one has $\nu_i=1$ and $M_i=0$.
The inequality $M_i\le\nu_i-1$ is proved in
Lemma~\ref{lem:chain_rule} below; it replaces, and reverses, the naive
expectation that differentiation lowers vanishing orders by exactly one,
which applies to the gradient of $f\circ\pi$ but \emph{not} to the pullback
of the ambient gradient $\nabla f$. When $\kappa=1$ we write simply $p^*$
and $\overline{p}$.

\subsection{Resolution versus log--resolution}
\label{subsec:resolution_vs_logresolution}

We begin by clarifying the geometric framework needed to quantify integrability,
starting with a distinction that is essential for the subsequent analysis.

Resolution of singularities plays a central role in the geometric analysis of
singular kernels. However, for the purposes of Morrey admissibility, it is
important to distinguish between resolving the singularities of the underlying
space and resolving a pair consisting of the space together with analytic data.
This distinction becomes essential once integrability and volume growth are
taken into account.

\paragraph{Ordinary resolution.}
An ordinary resolution of a real or complex analytic germ $(X,0)$ is a proper
birational morphism
\[
\pi:\widetilde X \longrightarrow X
\]
such that $\widetilde X$ is smooth, $\pi$ is an isomorphism over the regular locus
of $X$, and the exceptional divisor
\[
\Exc(\pi)=\pi^{-1}(\mathrm{Sing}(X))
\]
has simple normal crossings in $\widetilde{X}$.

Ordinary resolutions remove singularities of the space and provide detailed
information on its local geometry. However, they do not record how analytic
functions or measures vanish or concentrate near the exceptional divisor.
Consequently, ordinary resolutions alone do not yield quantitative information
sufficient for local integrability estimates.

\begin{example}[Why ordinary resolution is insufficient]
\label{ex:ordinary_resolution_insufficient}
Consider the plane curve singularity $X = \{xy = 0\} \subset \mathbb{R}^2$.
An ordinary resolution $\pi:\widetilde{X} \to X$ yields two smooth branches
meeting transversally at the exceptional divisor. While this resolution
clarifies the topological structure of $X$, it provides no information on
how a given analytic function $f = xy$ vanishes along the exceptional
components.

Now let $K = 1/f = 1/(xy)$ be a singular interaction kernel. To determine
whether $\nabla K \in L^p_{\mathrm{loc}}$, one must know not only that
$f$ vanishes along $E$, but also \emph{how fast} it vanishes---that is,
the vanishing order $\nu_{E}(f)$. This numerical data is invisible in an
ordinary resolution but becomes explicit in a log--resolution of the pair
$(X, f)$. Only the latter provides the quantitative input needed to
determine the admissible range of Morrey exponents.
\end{example}

\paragraph{Log--resolution.}
A log--resolution resolves both the space and prescribed analytic data. More
precisely, a log--resolution of a pair $(X,\mathcal I)$, where $\mathcal I$ is a
coherent analytic ideal, is a resolution
\[
\pi:\widetilde X \longrightarrow X
\]
such that $\widetilde X$ is smooth, the exceptional divisor has simple normal
crossings, and the total transform of $\mathcal I$ is locally monomial along the
exceptional locus.

Concretely, in adapted local coordinates $(y_1,\ldots,y_n)$ on $\widetilde{X}$,
the pullback ideal factors as
\[
\mathcal I\cdot\mathcal O_{\widetilde X}
=
(u) \cdot \prod_{i=1}^r (y_i)^{\nu_i},
\]
where $u$ is an analytic unit (nowhere vanishing), each hypersurface
$E_i = \{y_i = 0\}$ is an irreducible component of the exceptional divisor,
and $\nu_i := \nu_{E_i}(\mathcal I) \in \mathbb{N}$ denotes the
\emph{vanishing order} of $\mathcal{I}$ along $E_i$. Equivalently, in
divisorial notation,
\[
\mathcal I\cdot\mathcal O_{\widetilde X}
=
\mathcal O_{\widetilde X}\!\left(-\sum_{i=1}^r \nu_i\, E_i\right).
\]

Simultaneously, the resolution map induces a distortion of volume. The
Jacobian determinant of $\pi$ satisfies
\[
|\det J_\pi(y)|
=
|v(y)| \cdot \prod_{i=1}^r |y_i|^{a_i},
\]
where $v$ is again an analytic unit and $a_i \geq 0$ is the
\emph{discrepancy exponent} along $E_i$. In other words, the ambient
volume form pulls back as
\[
\pi^*(dx)
=
v(y) \cdot \prod_{i=1}^r |y_i|^{a_i} \, dy,
\]
where $dx$ denotes the Euclidean volume form on $X$ and $dy$ the
corresponding form on $\widetilde{X}$.

These numerical pairs $(\nu_i, a_i)$ encode the information needed to analyze
the local integrability of $|f|^{-\kappa}$ near the singular set. To treat
the gradient $\nabla K$, however, one further numerical invariant is needed:
the vanishing of $|\nabla f|$ along each divisor. This is measured by the
\emph{Jacobian ideal}
\[
\mathcal J_f := (\partial_1 f,\dots,\partial_n f)\subset\mathcal O_{X,0}.
\]
By Hironaka's theorem we may, and throughout the paper do, choose $\pi$ to be
a simultaneous log-resolution of the product ideal
$(f)\cdot\mathcal J_f$, so that in each adapted chart the pullback of
$\mathcal J_f$ is also principal and monomial:
\[
\mathcal J_f\cdot\mathcal O_{\widetilde X}
=
\Bigl(\prod_{i=1}^r y_i^{M_i}\Bigr),
\qquad
M_i:=\nu_{E_i}(\mathcal J_f)\in\mathbb Z_{\ge0}.
\]
Since the generators of the pulled-back ideal are the monomial times analytic
functions without common zeros, this principalization yields the two-sided
comparability
\begin{equation}\label{eq:gradf_monomial}
|\nabla f|\circ\pi
\;\asymp\;
\prod_{i=1}^r |y_i|^{M_i}
\qquad\text{locally on each chart,}
\end{equation}
which is precisely the estimate needed below. We emphasize that
\eqref{eq:gradf_monomial} concerns the pullback of the \emph{ambient}
gradient, not the gradient of $f\circ\pi$; the two are related by the chain
rule, which yields the following key inequality.

\begin{lemma}[Chain-rule bound for gradient orders]
\label{lem:chain_rule}
With the notation above, for every component $E_i$ along which $f\circ\pi$
vanishes,
\[
0\;\le\;M_i\;\le\;\nu_i-1.
\]
Moreover, if $\nabla f(0)=0$, then $M_i\ge1$ for every divisor $E_i$ with
$\pi(E_i)=\{0\}$. Locally on each component: $M_i=0$ forces $\nu_i=1$, and
on a strict-transform component $E_i$ with $M_i=0$ the ambient gradient
$\nabla f$ is generically nonzero along $\pi(E_i)$ (equivalently,
$\{f=0\}$ is reduced and smooth at the generic point of $E_i$).
\end{lemma}

\begin{proof}
By the chain rule,
\[
\nabla(f\circ\pi)(y)
=
D\pi(y)^{T}\,\bigl((\nabla f)\circ\pi\bigr)(y).
\]
The entries of $D\pi$ are analytic, hence have nonnegative order along each
$E_i$, so
\[
\operatorname{ord}_{E_i}\nabla(f\circ\pi)\;\ge\;
\operatorname{ord}_{E_i}\bigl((\nabla f)\circ\pi\bigr)
= M_i .
\]
On the other hand, differentiating the monomial factorisation
$f\circ\pi=u\prod_j y_j^{\nu_j}$ with respect to $y_i$ gives
\[
\partial_{y_i}(f\circ\pi)
=
\Bigl(\nu_i\,u + y_i\,\partial_{y_i}u\Bigr)
\prod_{j\ne i} y_j^{\nu_j}\; y_i^{\nu_i-1},
\]
whose restriction to $E_i=\{y_i=0\}$ equals
$\nu_i\,u\prod_{j\ne i}y_j^{\nu_j}$, not identically zero. Hence
$\operatorname{ord}_{E_i}\nabla(f\circ\pi)\le\nu_i-1$, and combining the two
displays gives $M_i\le\nu_i-1$.

If $\nabla f(0)=0$ and $\pi(E_i)=\{0\}$, then $(\nabla f)\circ\pi$ vanishes
identically on $E_i$, so $M_i\ge1$. Finally, $M_i=0$ means that
$\mathcal J_f$ does not vanish along $E_i$, i.e.\ $|\nabla f|$ is generically
nonzero along $\pi(E_i)\subset\{f=0\}$; since $M_i\le\nu_i-1$ forces
$\nu_i=1$ in that case, $E_i$ is a reduced strict-transform component along
which $\{f=0\}$ is generically smooth.
\end{proof}

The inequality $M_i\le\nu_i-1$ guarantees in particular that
$(\kappa+1)\nu_i-M_i\ge\kappa\nu_i+1>0$ for every $\kappa>0$, so the
threshold ratios introduced below are always well defined and positive.

\begin{remark}
\label{rem:pde_readers_logres}
For a reader focused on the PDE side, a log--resolution is simply a change
of coordinates in which $f$, its gradient magnitude $|\nabla f|$, and the
volume distortion $|\det J_\pi|$ all become explicit monomials near the
singularity. Away from the zero set of $f$, the integrability of
\[
|\nabla K| \;=\; \kappa\,\frac{|\nabla f|}{|f|^{\kappa+1}}
\qquad\text{for } K=|f|^{-\kappa}
\]
then reduces, via the change-of-variables formula, to the convergence of
one-dimensional monomial integrals
\[
\int_0^\varepsilon |y_i|^{-p\bigl((\kappa+1)\nu_i - M_i\bigr) + a_i} \, dy_i,
\]
where $M_i$ is the vanishing order of the Jacobian ideal along $E_i$. Each
such integral converges precisely when its exponent satisfies the
corresponding divisorial inequality, and these inequalities are exactly the
admissibility conditions on the Morrey exponents derived below.
\end{remark}

Concretely, this quantitative information is encoded by the numerical pairs
$(\nu_{E_i}(\mathcal I),a_{E_i})$ arising from a log--resolution, which
govern the local integrability of $\nabla K$ in the Morrey fixed--point
framework.

Although a log--resolution yields a finite family of divisorial ratios
\[
\frac{a_{E_i} + 1}{(\kappa+1)\nu_{E_i} - M_{E_i}},
\quad i = 1, \ldots, r,
\]
attached to the components of the total transform of $(f=0)$---exceptional
divisors and strict-transform components alike---only the minimal one
controls the leading--order concentration of volume near the singular set and
therefore determines the maximal integrability range relevant for Morrey
estimates and well--posedness of the associated aggregation equation. This
observation motivates the systematic use of log--resolutions throughout the
remainder of this section.

\subsection{Resolution data and asymptotic volume growth}
\label{subsec:resolution_data_volume}

We extract from a log--resolution
the numerical data that control both local volume growth and integrability near
the singular set.

Let $(X,0)$ be a real or complex analytic germ of pure dimension $n$, and let
$\mathcal I \subset \mathcal O_{X,0}$ be a coherent analytic ideal.
Fix a log--resolution
\[
\pi:\widetilde X \longrightarrow X
\]
of $\mathcal I$, and denote by
\[
\Exc(\pi)=\bigcup_{E\in\mathcal E} E
\]
the exceptional divisor.

In suitable local coordinates
$u=(u_1,\ldots,u_n)$ on $\widetilde X$ adapted to the divisor, by definition
of log--resolution, the total transform of the ideal is monomial:
\[
\mathcal I\cdot\mathcal O_{\widetilde X}
=
\prod_{E\in\mathcal E} u_E^{\nu_E(\mathcal I)},
\]
where each component $E$ is locally given by $u_E=0$ and
$\nu_E(\mathcal I)$ denotes the divisorial valuation of $\mathcal I$ along $E$.

\medskip
At the same time, the resolution map induces a distortion of volume.
If $dx$ denotes a smooth volume form on $X$, then its pullback satisfies
\[
\pi^*(dx)\;\asymp\;\prod_{E\in\mathcal E}|u_E|^{a_E}\,du,
\]
where $a_E$ is the discrepancy exponent of $E$ and $du$ denotes the Euclidean
volume form in the resolution coordinates.

These two families of exponents play complementary roles:

\begin{itemize}
\item $\nu_E(\mathcal I)$ measures the vanishing of the ideal along $E$;
\item $a_E$ measures the Jacobian distortion of the resolution map.
\end{itemize}

Together, they determine the asymptotic behavior of integrals of the form
\[
\int_{U} \Phi(x)\,dx
\]
for functions $\Phi$ with singularities along $V(\mathcal I)$.
After resolution and change of variables, such integrals reduce locally to
products of one--dimensional monomial integrals
\[
\int |u_E|^{\alpha_E}\,du_E,
\]
whose convergence is organized by the exponents $\alpha_E>-1$.
This reduction is the basic mechanism behind the integrability criteria
used throughout the paper.

\medskip
In direct analogy with the linear and bilinear estimates in
Section~\ref{sec:expanded_background}, the analysis therefore focuses on:

\begin{enumerate}
\item local $L^p$ properties of $\nabla K$ near the singular set,
\item the role of resolution data in determining admissible exponents,
\item convolution estimates for the operator
\[
T_K(u)=\nabla K*u,
\]
which plays the same structural role as the gradient of the heat semigroup in
the classical aggregation theory.
\end{enumerate}

The guiding principle is that the admissible Morrey exponents are ultimately
controlled by the same divisorial data that govern the integrability of
reciprocal analytic functions and the asymptotic volume of sublevel sets.

\begin{example}[Monomial calculation: the cusp]
\label{ex:cusp}
Consider the analytic function $f(x,y)=x^2+y^3$ and the kernel $K=1/f$
(i.e.\ $\kappa=1$). A log--resolution is obtained from the weighted blow--up
with weights $(3,2)$; in the chart $x=t^3$, $y=t^2 s$ one computes
\[
f\circ\pi = t^{6}\,(1+s^{3}),
\qquad
|\det J_\pi| = 3\,|t|^{4},
\qquad
(\nabla f)\circ\pi = \bigl(2t^{3},\,3t^{4}s^{2}\bigr),
\]
so that along the exceptional divisor $E=\{t=0\}$,
\[
\nu_E(f)=6,
\qquad
a_E=4,
\qquad
M_E=\nu_E(\mathcal J_f)=3.
\]
Note that $M_E=3$ is \emph{strictly smaller} than $\nu_E-1=5$: the value $5$
is the vanishing order of $\nabla(f\circ\pi)=\bigl(6t^{5}(1+s^{3})+\cdots\bigr)$,
i.e.\ of the gradient computed \emph{in the resolution coordinates}, which is a
different object from the pullback of the ambient gradient
(cf.\ Lemma~\ref{lem:chain_rule}). The exceptional divisor therefore
contributes the ratio
\[
\frac{a_E+1}{2\nu_E - M_E}
=\frac{5}{12-3}
=\frac{5}{9}.
\]

However, the \emph{real} zero set of $f$ is not the origin alone: it is the
cusp curve $x^2=-y^3$, $y\le0$. Its strict transform $E_0$ carries the data
$\nu_{E_0}=1$, $M_{E_0}=0$, $a_{E_0}=0$ (the curve is reduced and generically
smooth), hence the ratio $1/(2\cdot1-0)=1/2$. The threshold is the minimum of
the two contributions:
\[
\nabla K\in L^p_{\mathrm{loc}}
\quad\Longleftrightarrow\quad
p<\min\Bigl(\tfrac59,\tfrac12\Bigr)=\tfrac12 .
\]

Both contributions can be verified directly. Near a smooth point of the
curve, $|f|\asymp d$ and $|\nabla f|\asymp1$, where $d$ is the distance to the
curve, so $|\nabla K|\asymp d^{-2}$ and $\int d^{-2p}\,dd<\infty$ iff $p<1/2$.
Near the origin, the anisotropic scaling $x=\varepsilon^3 X$,
$y=\varepsilon^2 Y$ on dyadic annuli gives $|f|\sim\varepsilon^{6}$,
$|\nabla f|\sim\varepsilon^{3}$, hence $|\nabla K|\sim\varepsilon^{-9}$ on a
region of volume $\sim\varepsilon^{5}$, and
$\sum_k 2^{-k(5-9p)}<\infty$ iff $p<5/9$. This simple calculation
illustrates how the full resolution data---including the gradient orders
$M_i$ and the strict transform---determine the critical integrability
threshold, and how omitting either ingredient produces an incorrect value.
\end{example}

\subsection{The real log--canonical threshold and the gradient integrability threshold}
\label{subsec:rlct_definition}

With the resolution data in place, we are now in a position to introduce the
geometric invariant that governs the dominant integrability behavior of analytic
singularities.
Rather than providing an \emph{a priori} exact characterization of the
integrability threshold, this invariant encodes the leading geometric
contribution to the local concentration of volume and therefore yields a
robust geometric criterion for the class of singularities under study.

\begin{remark}[Real resolutions]
\label{rem:real_resolution}
By ``log-resolution'' we always mean a resolution in the category of
real-analytic manifolds; in the isolated-singularity setting considered
here, its existence follows from Hironaka's theorem together with its
real-analytic refinements (Bierstone--Milman). A subtler point is whether a
log-resolution defined over $\mathbb R$ actually restricts, at every real
point, to an honest $\mathbb R$-normal-crossing local model, rather than
merely an ambient complex simple-normal-crossing structure viewed through
its real points. This is not automatic; it is established rigorously in
\cite[Thm.~3.1]{kostawindisch2026}, via an isomorphism of completions of
local rings and Cohen's structure theorem. The divisorial formula
$\mathrm{rlct}_0(f)=\min_i(a_i+1)/\nu_i$ used below
(Remark~\ref{rem:rlct_divisorial}) is the real-analytic counterpart of the
formula obtained on that basis in \cite[Def./Thm.~3.4]{kostawindisch2026}.
Consequently, every divisorial component
$E_i$ appearing in the minima defining $\mathrm{rlct}_0(f)$,
$p^*(\kappa)$, and $\mathrm{MTI}(K)$ below is a real subvariety of
the real resolution space $\widetilde X$, computed
from the real vanishing orders $\nu_i$, $M_i$, and the real discrepancy
$a_i$: no complexification enters these minima, and no complex-only
exceptional component is ever included.
\end{remark}

\begin{definition}[Real log--canonical threshold]
Let $f$ be real--analytic in a neighbourhood of $0\in\mathbb{R}^n$ with $f(0)=0$.
The real log--canonical threshold of $f$ at the origin is defined by
\[
\mathrm{rlct}_0(f)
:=
\sup\bigl\{\,c>0:\;|f|^{-c}\in L^1_{\mathrm{loc}}\,\bigr\}.
\]
\end{definition}

The real log--canonical threshold provides a sharp measure of the local
integrability of the analytic germ $f$ itself and identifies the divisors
that dominate the asymptotic volume growth of the sublevel sets
$\{|f|\le\varepsilon\}$ as $\varepsilon\to0$.
As such, it captures the principal geometric scale at which singular behavior
concentrates near the origin.

\begin{remark}[Divisorial formula for the RLCT]
\label{rem:rlct_divisorial}
In terms of a real log-resolution as above (Remark~\ref{rem:real_resolution}),
the standard change-of-variables computation gives
\[
\mathrm{rlct}_0(f)=\min_i\frac{a_i+1}{\nu_i},
\]
where the minimum runs over \emph{all} components $E_i$ of the total
transform of $(f=0)$ meeting $\pi^{-1}(0)$, including strict-transform
components---all of which are real, by Remark~\ref{rem:real_resolution}.
When the real zero set of $f$ contains a hypersurface through
the origin, a reduced strict-transform component contributes the ratio
$(0+1)/1=1$, so in that case $\mathrm{rlct}_0(f)\le1$, with equality if and
only if every other divisor $E_i$ entering the minimum also satisfies
$(a_i+1)/\nu_i\ge1$ (e.g.\ when $f$ is smooth with $\nabla f(0)\neq0$, so
that this reduced component is the only one meeting $\pi^{-1}(0)$). If
instead the real zero set is of higher codimension---for instance an
isolated zero---the RLCT can exceed $1$: e.g.\
$\mathrm{rlct}_0(|x|^2)=d/2$ in $\mathbb R^d$. We emphasize that a
component of the strict transform is included in this minimum only when it
is a reduced, generically smooth, \emph{real} branch of $\{f=0\}$ actually
meeting the fiber $\pi^{-1}(0)$; components without real points, or whose
real locus does not meet $\pi^{-1}(0)$, do not contribute. This bookkeeping
is essential rather than a formality: $\lambda_{\mathbb R}\neq\lambda_{\mathbb
C}$ can genuinely occur, as shown by the explicit examples in
\cite[Ex.~5.5, 5.6]{kostawindisch2026}, whose Lemmas~3.5--3.6 give sufficient
conditions, in terms of which divisors meet the real locus, for the real and
complex thresholds (and their multiplicities) to coincide.
\end{remark}

\begin{definition}[Gradient Integrability Threshold]\label{def:MTI}
Let $K$ be an interaction kernel defined in a neighbourhood of the origin in
$\mathbb R^n$. The Gradient Integrability Threshold of $K$ is defined by
\[
\mathrm{MTI}(K)
:=\sup\bigl\{\,p>0:\ \nabla K\in L^p_{\mathrm{loc}}(\mathbb R^n)\,\bigr\}.
\]
\end{definition}

\begin{remark}[Scope and conventions]
\label{rem:MTI_scope}
The quantity $\mathrm{MTI}(K)$ may take the value $+\infty$. If $K$ is smooth
near the origin, then $\nabla K\in L^p_{\mathrm{loc}}$ for every $p>0$, and
consequently $\mathrm{MTI}(K)=\infty$. The interest of the present work lies
in the singular case, where the index is finite. We emphasize that
$\mathrm{MTI}(K)$ is, by definition, a statement about local $L^p$
integrability of $\nabla K$ alone; its relevance to the Morrey fixed-point
framework of Section~\ref{sec:morrey_framework} is made precise only in
Section~\ref{sec:morrey_bridge} below, where it is identified with a
specific inequality involving the Morrey scaling parameter.

Under the two-sided comparability assumption
$|\nabla K|\asymp |\nabla f|/|f|^{\kappa+1}$
of Hypothesis~\ref{hyp:dominant} below, the exact $L^p_{\mathrm{loc}}$
threshold of $\nabla K$ is the resolution-theoretic quantity
$p^*(\kappa)=\min_i \frac{a_i+1}{(\kappa+1)\nu_i-M_i}$
arising from a log-resolution of $f$ and its Jacobian ideal, while the real
log-canonical threshold provides a computable geometric lower bound via
$p < \mathrm{rlct}_0(f)/(\kappa+1)$.
The precise relationship between these quantities is established in
Theorem~\ref{thm:admissibility} below; in particular, the lower bound is
\emph{strict} whenever the threshold is governed by divisors lying over a
critical point of $f$.
\end{remark}

The index $\mathrm{MTI}(K)$ is a purely local, purely geometric quantity. Its
role as an \emph{input} to the fixed--point framework for aggregation
equations---specifically, as the threshold controlling for which Morrey
scaling parameters the local integrability hypothesis on $\nabla K$ used in
\cite{suleiman2020aggregation,suleiman2023existence} is satisfied---is made
precise in Section~\ref{sec:morrey_bridge}.

\begin{remark}
Under the two-sided comparability assumption
$|\nabla K|\asymp |\nabla f|/|f|^{\kappa+1}$
(Hypothesis~\ref{hyp:dominant}), the Gradient Integrability Threshold equals
\[
\mathrm{MTI}(K) = p^*(\kappa) = \min_i \frac{a_i+1}{(\kappa+1)\nu_i-M_i}
\]
for the divisorial data arising from a log-resolution of $f$ and its Jacobian
ideal. The real log-canonical threshold provides the geometric lower bound
$\mathrm{rlct}_0(f)/(\kappa+1) \le \mathrm{MTI}(K)$,
reflecting the dominant contribution to volume concentration, but---as
Theorem~\ref{thm:admissibility}\,(iii) makes precise---the two coincide only
in the mildly singular regime governed by the smooth part of the zero set,
and never when $f$ has an isolated zero in dimension $n\ge2$.
\end{remark}


\begin{hypothesis}[Dominant analytic singularity]
\label{hyp:dominant}
Let $K$ be a singular interaction kernel defined in a neighborhood of the
origin, and let $\kappa>0$. We assume that there exists a real-analytic
function $f$, not identically zero, with $f(0)=0$, such that, away from the
zero set of $f$ in a neighborhood of the origin,
\begin{equation}\label{eq:twoside}
|\nabla K(x)| \;\asymp\; \frac{|\nabla f(x)|}{|f(x)|^{\kappa+1}},
\end{equation}
i.e.\ there exist constants $0 < c_1 \le c_2 < \infty$ and $\varepsilon_0>0$
such that
\[
c_1\,\frac{|\nabla f(x)|}{|f(x)|^{\kappa+1}}
\;\le\;
|\nabla K(x)|
\;\le\;
c_2\,\frac{|\nabla f(x)|}{|f(x)|^{\kappa+1}}
\qquad\text{for all }x\in B_{\varepsilon_0}(0)\setminus\{f=0\}.
\]
The two-sided bound ensures that the $L^p$-integrability of $\nabla K$ and
that of $|\nabla f|\,|f|^{-(\kappa+1)}$ coincide, so that the admissibility
threshold is a genuine two-sided characterisation rather than merely a
sufficient condition.

The assumption \eqref{eq:twoside} is satisfied by the model kernels
$K=|f|^{-\kappa}$; more generally whenever $K = G(f)$ with
$|G'(s)| \asymp |s|^{-(\kappa+1)}$ near $s=0$, up to analytic multiplicative
units; whenever the singular set of $K$ coincides with the zero locus of an
analytic constraint function; and it is stable under local analytic
diffeomorphisms. The classical Newtonian kernel corresponds to $f=|x|^2$ and
$\kappa=(d-2)/2$; see Section~\ref{subsec:newtonian}.
\end{hypothesis}

\begin{remark}[Scope of Hypothesis~\ref{hyp:dominant}]
\label{rem:hyp_fail}
Since the flexibility in the exponent $\kappa$ covers all kernels of the
form $|f|^{-\kappa}$, the main situation genuinely outside the scope of
\eqref{eq:twoside} is that in which the singular behavior of $\nabla K$ is
not modeled by any single analytic divisor---for instance, kernels whose
gradient exhibits singular behavior of different analytic types along
different directions that cannot be absorbed into one function $f$, or
kernels involving non-analytic (e.g.\ logarithmic or oscillatory) factors.
In such cases a stratified or valuation-theoretic refinement of the present
framework would be required, and this falls outside the scope of this paper.
By contrast, hypersurface zero sets of $f$ pose no
difficulty: the strict transform of $\{f=0\}$ enters the divisorial minimum
of Theorem~\ref{thm:admissibility} on the same footing as the exceptional
divisors, as illustrated in Example~\ref{ex:cusp}.
\end{remark}

Under Hypothesis~\ref{hyp:dominant}, the exact $L^p_{\mathrm{loc}}$
threshold of $\nabla K$ is equivalent to
the analysis of the model singularity $|\nabla f|\,|f|^{-(\kappa+1)}$, with
the admissible range of exponents determined by resolution data associated
with the pair $(f,\mathcal J_f)$, as made precise in
Theorem~\ref{thm:admissibility} below; its consequence for Morrey
admissibility is drawn out separately in Section~\ref{sec:morrey_bridge}.
\medskip

\subsection*{A model non-radial singularity}

To illustrate how resolution data translate into concrete admissibility
thresholds, we consider a genuinely anisotropic example for which classical
radial estimates fail to provide sharp bounds.

Let
\[
K(x)=\bigl(x_1^4+x_2^6\bigr)^{-\kappa},
\qquad x=(x_1,x_2)\in\mathbb{R}^2,
\]
with $\kappa>0$, and set $f(x)=x_1^4+x_2^6$, so that $K=|f|^{-\kappa}$.
The real zero set of $f$ is the origin alone, the singularity is non-radial,
and it exhibits distinct scaling behavior along the coordinate directions.

A log--resolution of $f$ and its Jacobian ideal is obtained from the weighted
blow--up with weights $(3,2)$. In the chart $x_1=t^3$, $x_2=t^2 s$ one finds
\[
f\circ\pi=t^{12}\bigl(1+s^6\bigr),
\qquad
|\det J_\pi|=3\,|t|^{4},
\qquad
(\nabla f)\circ\pi=\bigl(4t^{9},\,6t^{10}s^{5}\bigr),
\]
so the resulting exceptional divisor $E=\{t=0\}$ satisfies
\[
\nu_E(f)=12,
\qquad
a_E=4,
\qquad
M_E=9 .
\]
Hence
\[
\mathrm{rlct}_0(f)=\frac{a_E+1}{\nu_E(f)}=\frac{5}{12},
\]
and, since
\[
|\nabla K|=\kappa\,\frac{|\nabla f|}{|f|^{\kappa+1}},
\]
Theorem~\ref{thm:admissibility} below yields the \emph{exact} threshold
\[
\nabla K\in L^p_{\mathrm{loc}}(\mathbb{R}^2)
\quad\text{iff}\quad
p<p^*(\kappa)=\frac{a_E+1}{(\kappa+1)\nu_E-M_E}=\frac{5}{12\kappa+3}.
\]
For $\kappa=1$ this gives $p^*=1/3$. Observe the two structural features
that the corrected framework makes visible. First, the threshold genuinely
depends on $\kappa$, as it must: a stronger pole yields a smaller
admissible range. Second, the RLCT bound
\[
\overline{p}(\kappa)=\frac{\mathrm{rlct}_0(f)}{\kappa+1}
=\frac{5}{12\kappa+12}
\;<\;\frac{5}{12\kappa+3}=p^*(\kappa)
\]
is \emph{strictly} smaller than the exact threshold, because the gradient
vanishes along $E$ ($M_E=9\ge1$); this strictness is a general phenomenon in
the singular regime (Theorem~\ref{thm:admissibility}\,(iii)).

Both values can be verified by the anisotropic scaling
$x_1=\varepsilon^3X$, $x_2=\varepsilon^2Y$ on dyadic annuli:
$|f|\sim\varepsilon^{12}$, $|\nabla f|\sim\varepsilon^{9}$, hence
$|\nabla K|\sim\varepsilon^{-(12\kappa+3)}$ on regions of volume
$\sim\varepsilon^{5}$, giving convergence exactly for
$(12\kappa+3)p<5$.

The inclusion $\nabla K\in L^p_{\mathrm{loc}}$ for $p<p^*(\kappa)$ provides
the local integrability input for the bilinear estimate
\[
(u,v)\longmapsto u\,(\nabla K*v)
\]
in the Morrey framework of
\cite{suleiman2020aggregation,suleiman2023existence}; see the discussion of
the two mechanisms below.
In contrast, isotropic Hardy--Littlewood--Sobolev estimates based on the
radial majorant $|\nabla K(x)|\le C\,|x|^{-(6\kappa+1)}$ (attained along the
$x_2$-axis) only predict the non-optimal range $p<2/(6\kappa+1)$, e.g.\
$p<2/7$ instead of $p<1/3$ when $\kappa=1$: resolution data recover the
sharp admissibility regime dictated by the geometry of the singularity.

\medskip

\subsubsection*{Local integrability via log resolutions}

Let $\pi:\widetilde X\to X$ be a log resolution of the pair
$(f,\mathcal J_f)$. In adapted local coordinates the total transform of
$(f=0)$ has normal crossings,
\[
\bigcup_i E_i=\{y_i=0\},
\]
and the pullbacks admit the two-sided monomial comparabilities
\[
|f\circ\pi|\asymp\prod_i |y_i|^{\nu_i}, \qquad
|J_\pi|\asymp\prod_i |y_i|^{a_i}, \qquad
|\nabla f|\circ\pi\asymp\prod_i |y_i|^{M_i},
\]
the last by the principalization of $\mathcal J_f$
(cf.\ \eqref{eq:gradf_monomial}), with
$0\le M_i\le \nu_i-1$ by Lemma~\ref{lem:chain_rule}.

Since $K=|f|^{-\kappa}$,
\[
|\nabla K|=\kappa\,\frac{|\nabla f|}{|f|^{\kappa+1}},
\]
and therefore
\[
|\nabla K|\circ\pi
\;\asymp\;
\prod_i |y_i|^{M_i-(\kappa+1)\nu_i}.
\]
Consequently,
\[
\bigl(|\nabla K|\circ\pi\bigr)^p\,|J_\pi|
\;\asymp\;
\prod_i |y_i|^{-p((\kappa+1)\nu_i-M_i)+a_i}.
\]

Because $\pi$ is proper and restricts to a diffeomorphism off a set of
measure zero, local integrability of $|\nabla K|^p$ near the origin reduces
to integrability of this monomial expression on finitely many chart
polydiscs covering $\pi^{-1}(0)$. By Fubini's theorem this is equivalent to
convergence of
\[
\int_0^\varepsilon |y_i|^{-p((\kappa+1)\nu_i-M_i)+a_i}\,dy_i,
\]
which holds iff
\[
-p\bigl((\kappa+1)\nu_i-M_i\bigr)+a_i>-1,
\quad\text{i.e.}\quad
p<\frac{a_i+1}{(\kappa+1)\nu_i-M_i}.
\]
Note that the denominators are positive since $M_i\le\nu_i-1$ implies
$(\kappa+1)\nu_i-M_i\ge\kappa\nu_i+1$.

\subsubsection*{Two complementary mechanisms controlling convolution}

There are two analytic mechanisms controlling the operator \(T_K\). Only the
first is the one actually used by the well-posedness theorem of
\cite{suleiman2020aggregation,suleiman2023existence} (Lemma~\ref{lem:bilinear}
above); the second is recorded for completeness as a general fact about
convolution in Morrey spaces, valid on its own terms, but it is \emph{not}
the hypothesis invoked by that reference. Throughout, we assume $K$ is
specified, via Hypothesis~\ref{hyp:dominant}, on a neighborhood
$B_{\varepsilon_0}(0)$ of the singularity; for the global operator $T_K$ on
$\mathbb R^n$ one further assumes $K$ is defined and smooth outside a
compact set containing $B_{\varepsilon_0}(0)$, so that only the local
behavior near $0$ is governed by the geometric analysis of this section.

\medskip
\noindent\textbf{(A) Global $L^1$ criterion (the hypothesis of \cite{suleiman2020aggregation,suleiman2023existence}).}
Decompose $\nabla K=\chi_{B_\varepsilon}\nabla K+(1-\chi_{B_\varepsilon})\nabla K$
for $\varepsilon<\varepsilon_0$. The far part is bounded and smooth on
$B_{\varepsilon_0}(0)\setminus B_\varepsilon(0)$ by Hypothesis~\ref{hyp:dominant},
and is bounded with compact support outside $B_{\varepsilon_0}(0)$ by the
global assumption above; in particular it is a finite measure. The
divisorial characterisation of Theorem~\ref{thm:admissibility} gives a
computable criterion for the near part to also be a finite measure: one has
$\nabla K\in L^1_{\mathrm{loc}}$ near the origin if and only if
\[
1<p^*(\kappa)=\min_i\frac{a_i+1}{(\kappa+1)\nu_i-M_i},
\qquad\text{i.e.}\qquad
a_i\ge(\kappa+1)\nu_i-M_i\ \text{ for all } i.
\]
When $p^*(\kappa)>1$, both pieces are finite measures, so
$\nabla K\in L^1(\mathbb R^n)$ globally, with
\[
\|\nabla K\|_{L^1(\mathbb R^n)}
=\|\chi_{B_\varepsilon}\nabla K\|_{L^1}+\|(1-\chi_{B_\varepsilon})\nabla K\|_{L^1}
<\infty,
\]
and Lemma~\ref{lem:young_morrey}(i)--(ii) yields
$\|\nabla K*u\|_{M_{p,\lambda}}\le\|\nabla K\|_{L^1}\|u\|_{M_{p,\lambda}}$.
This is precisely the hypothesis $\nabla K\in L^1$ of
\cite[Thm.~3.1]{suleiman2023existence}, and Lemma~\ref{lem:bilinear} above
is the resulting bilinear estimate, with constants explicit multiples of
$\|\nabla K\|_{L^1(\mathbb R^n)}$.

Conversely, if $p^*(\kappa)\le1$, then $\nabla K\notin L^1_{\mathrm{loc}}$
near the origin (Theorem~\ref{thm:admissibility}(i) with $p=1$), so
\emph{no} choice of global truncation or decay of $K$ away from the
singularity can repair $\nabla K\in L^1(\mathbb R^n)$: local failure at the
singularity is an obstruction that no modification at infinity can cure.
For such kernels the hypothesis of \cite[Thm.~3.1]{suleiman2023existence}
fails, independently of the Morrey parameters $(p,\lambda)$ of the solution
space.

\medskip
\noindent\textbf{(B) Local $L^p_{\mathrm{loc}}$ estimates (general fact, not
used below).} For any $1\le p<p^*(\kappa)$, the local norms
$\|\nabla K\|_{L^p(B_\varepsilon)}$ are finite by
Theorem~\ref{thm:admissibility}(i), and a standard Hölder--Young argument in
Morrey spaces (Lemma~\ref{lem:holder_morrey} together with a
Hardy--Littlewood--Sobolev-type convolution estimate) would in principle
bound $\nabla K*u$ in terms of such local $L^p$ norms for $p\ne1$. We record
this only to note that it is \emph{not} what
\cite{suleiman2020aggregation,suleiman2023existence} require: their
Theorem~3.1 and its bilinear estimate (Lemma~\ref{lem:bilinear} above) use
exclusively the global $L^1$ estimate of mechanism (A). Since we do not use
mechanism (B) in what follows, we do not develop it further.

\begin{remark}[Two regimes]
\label{rem:three_regimes}
In view of mechanism (A), the relevant trichotomy collapses to:
\begin{itemize}
\item[\textbf{Regime A}] ($p^*(\kappa)>1$): $\nabla K\in L^1_{\mathrm{loc}}$
near the singularity, and---once $K$ is modified to be smooth and
compactly supported (or otherwise globally integrable in gradient) away
from a neighborhood of the singularity---$\nabla K\in L^1(\mathbb R^n)$,
so the hypothesis of \cite[Thm.~3.1]{suleiman2023existence} is satisfied.
This is the regime made precise in Section~\ref{sec:morrey_bridge}.
\item[\textbf{Regime C}] ($p^*(\kappa)\le1$): $\nabla K\notin L^1_{\mathrm{loc}}$
near the singularity, and no global modification can repair this; the
kernel falls outside the scope of \cite{suleiman2020aggregation,suleiman2023existence}
as formulated, and would require a different functional framework
(e.g.\ a genuinely fractional or measure-theoretic reformulation of the
bilinear estimate).
\end{itemize}
Crucially, this dichotomy is governed by $p^*(\kappa)$ \emph{alone}: it does
not involve the Morrey parameters $(p,\lambda)=(n-\lambda,\lambda)$ of the
solution space at all, since those enter only through the choice of the
space $M_{p,\lambda}$ in which the initial data $u_0$ is posed
(Section~\ref{sec:morrey_framework}), not through the integrability
requirement on $\nabla K$.
\end{remark}

Thus the real log--canonical threshold identifies the dominant geometric scale
governing integrability, while the exact $L^p_{\mathrm{loc}}$ threshold for
the velocity field is determined by the finer resolution--theoretic data
appearing in the ratios $\frac{a_i+1}{(\kappa+1)\nu_i-M_i}$.

\medskip

We now state the main result, which gives a complete geometric
characterisation of the exact $L^p_{\mathrm{loc}}$ integrability threshold
of $\nabla K$ together with the role of the RLCT; its consequence for
Morrey admissibility proper is drawn out in
Corollary~\ref{cor:morrey_bridge} of Section~\ref{sec:morrey_bridge}.

\begin{theorem}[Exact gradient-integrability threshold]
\label{thm:admissibility}
Let $f$ be a real-analytic function on a neighborhood of $0\in\mathbb R^n$,
not identically zero, with $f(0)=0$, let $\kappa>0$, and set
$K=|f|^{-\kappa}$. Let $\pi:\widetilde X\to X$ be a log-resolution of the
pair $(f,\mathcal J_f)$ over a neighborhood of the origin, and denote by
$(\nu_i, M_i, a_i)$ the vanishing orders of $f$, the vanishing orders of the
Jacobian ideal $\mathcal J_f$, and the discrepancy exponents along the
components $E_i$ of the total transform of $(f=0)$ meeting $\pi^{-1}(0)$
(exceptional divisors and strict-transform components alike). Recall from
Lemma~\ref{lem:chain_rule} that $0\le M_i\le\nu_i-1$.

\begin{enumerate}
\item[\textup{(i)}] \textbf{Exact admissibility characterisation.}
\[
\nabla K \in L^p_{\mathrm{loc}}
\quad\Longleftrightarrow\quad
p < p^*(\kappa) := \min_i \frac{a_i+1}{(\kappa+1)\nu_i-M_i}.
\]
In particular, $\mathrm{MTI}(K) = p^*(\kappa)$.

\item[\textup{(ii)}] \textbf{RLCT bounds.} Using $0\le M_i\le\nu_i-1$
divisor by divisor,
\[
\overline{p}(\kappa)
:=\frac{\mathrm{rlct}_0(f)}{\kappa+1}
=\frac{1}{\kappa+1}\min_i\frac{a_i+1}{\nu_i}
\;\le\;
p^*(\kappa)
\;\le\;
\min_i\frac{a_i+1}{\kappa\nu_i+1},
\]
so that $p < \overline{p}(\kappa)$ yields a computable sufficient condition
for $\nabla K\in L^p_{\mathrm{loc}}$. In particular,
$\overline{p}(\kappa)>1$ (equivalently $\kappa+1<\mathrm{rlct}_0(f)$) gives a
computable sufficient condition for $p^*(\kappa)>1$, i.e.\ for Regime A of
Remark~\ref{rem:three_regimes}, which is what feeds into Morrey
admissibility via Corollary~\ref{cor:morrey_bridge}.

\item[\textup{(iii)}] \textbf{Strictness of the RLCT bound.}
Equality $p^*(\kappa)=\overline{p}(\kappa)$ holds if and only if some
divisor $E_i$ realizing $\mathrm{rlct}_0(f)$ satisfies $M_i=0$. By
Lemma~\ref{lem:chain_rule}, on any such component
$M_i=0\iff\nu_i=1$ and $\mathcal J_f$ is generically nonvanishing along
$E_i$; hence $E_i$ is necessarily a reduced, generically smooth
strict-transform component. In particular:
\begin{itemize}
\item if $\nabla f(0)\neq0$ (so that $\{f=0\}$ is a smooth hypersurface near
$0$ after reduction), then
$p^*(\kappa)=\overline{p}(\kappa)=\tfrac{1}{\kappa+1}$;
\item if the minimum defining $\mathrm{rlct}_0(f)$ is attained only by
divisors along which $\mathcal J_f$ vanishes---in particular, whenever $f$
has an isolated zero at the origin and $n\ge2$, since then $\nabla f(0)=0$
and every real divisorial component $E_i$ meeting $\pi^{-1}(0)$ and
contributing to the divisorial minima defining $p^*(\kappa)$ and
$\mathrm{rlct}_0(f)$ (Remark~\ref{rem:real_resolution}) necessarily maps to
the origin---the inequality is strict:
$p^*(\kappa)>\overline{p}(\kappa)$.
\end{itemize}
\end{enumerate}

Under Hypothesis~\ref{hyp:dominant}, part~\textup{(i)} gives the exact
$L^p$-threshold of the original kernel $K$ in place of the model
$|f|^{-\kappa}$. The inclusion $\nabla K\in L^p_{\mathrm{loc}}$ for
exponents $p < p^*(\kappa)$ provides the local $L^p$-integrability input
required by the well-posedness framework of
\cite{suleiman2020aggregation,suleiman2023existence}, as discussed in
mechanisms (A)--(B) above.
\end{theorem}

\begin{proof}
By construction, in each local chart of $\widetilde X$ there exist adapted
local coordinates $(y_1,\dots,y_n)$ such that the pullbacks of $f$ and of the
Jacobian determinant of $\pi$ admit monomial factorizations, and the pullback
of the Jacobian ideal $\mathcal J_f$ is principal and monomial:
\[
f\circ\pi
=
u(y)\prod_{i=1}^r y_i^{\nu_i},
\qquad
|J_\pi(y)|
=
v(y)\prod_{i=1}^r |y_i|^{a_i},
\qquad
\mathcal J_f\cdot\mathcal O_{\widetilde X}
=
\Bigl(\prod_{i=1}^r y_i^{M_i}\Bigr),
\]
where $u,v$ are analytic units, $\nu_i\in\mathbb N$ are the vanishing orders
of $f$ along the components $E_i=\{y_i=0\}$ of the total transform,
$a_i\in\mathbb Z_{\ge0}$ are the corresponding discrepancy exponents (with
$a_i=0$ on strict-transform components), and $M_i=\nu_{E_i}(\mathcal J_f)$.
Since the generators $(\partial_j f)\circ\pi$ of the pulled-back ideal equal
the monomial times analytic functions without common zeros, we have the
two-sided comparability \eqref{eq:gradf_monomial},
\[
|\nabla f|\circ\pi
\;\asymp\;
\prod_{i=1}^r |y_i|^{M_i}
\qquad\text{locally on each chart,}
\]
and Lemma~\ref{lem:chain_rule} gives
\[
0\le M_i \le \nu_i-1
\qquad\text{for all } i,
\]
so that $(\kappa+1)\nu_i-M_i\ge\kappa\nu_i+1>0$ and all the ratios below
are well defined.

\medskip
\noindent\emph{Step 1: Proof of part~\textup{(i)}.}
Since $K=|f|^{-\kappa}$, away from $\{f=0\}$,
\[
|\nabla K| = \kappa\,\frac{|\nabla f|}{|f|^{\kappa+1}}.
\]
Pulling back by $\pi$ and using the monomial comparabilities yields, in each
chart,
\[
|\nabla K|\circ\pi
\;\asymp\;
\frac{\prod_i |y_i|^{M_i}}{\prod_i |y_i|^{(\kappa+1)\nu_i}}
=
\prod_{i=1}^r |y_i|^{M_i-(\kappa+1)\nu_i}.
\]
Consequently, for any $p>0$,
\[
\bigl(|\nabla K|\circ\pi\bigr)^p\,|J_\pi|
\;\asymp\;
\prod_{i=1}^r |y_i|^{-p((\kappa+1)\nu_i-M_i)+a_i}.
\]

Let $U\subset X$ be a sufficiently small neighborhood of the origin. Since
$\pi$ is proper, $\pi^{-1}(\overline U)$ is compact and may be covered by
finitely many chart polydiscs, and $\pi$ restricts to a diffeomorphism off a
set of measure zero. By the change of variables formula, $\int_U |\nabla
K|^p\,dx$ is finite if and only if the integral of the displayed monomial
expression converges on each chart polydisc. Since the divisor has simple
normal crossings, Fubini's theorem bounds each chart integral above and
below by constant multiples of
\[
\prod_{i=1}^r \int_0^\varepsilon |y_i|^{-p((\kappa+1)\nu_i-M_i)+a_i}\,dy_i,
\]
and each factor converges if and only if
$-p((\kappa+1)\nu_i-M_i)+a_i>-1$,
i.e.\ $p<(a_i+1)/\bigl((\kappa+1)\nu_i-M_i\bigr)$. For the divergence
direction, if $p\ge(a_j+1)/\bigl((\kappa+1)\nu_j-M_j\bigr)$ for some $j$,
one restricts the integral to a compact region of a chart meeting $E_j$ on
which the comparability constants are uniform and the remaining variables
range over a fixed compact set, and the $y_j$-integral diverges.
Taking the minimum over $i$ completes the proof of part~(i).

\medskip
\noindent\emph{Step 2: Proof of part~\textup{(ii)}.}
Both inequalities follow divisor by divisor from
$0\le M_i\le\nu_i-1$. From $M_i\ge0$,
\[
\frac{a_i+1}{(\kappa+1)\nu_i-M_i}
\;\ge\;
\frac{a_i+1}{(\kappa+1)\nu_i}
=
\frac{1}{\kappa+1}\,\frac{a_i+1}{\nu_i},
\]
and taking minima over $i$ gives
$p^*(\kappa)\ge\overline{p}(\kappa)$. From $M_i\le\nu_i-1$,
\[
\frac{a_i+1}{(\kappa+1)\nu_i-M_i}
\;\le\;
\frac{a_i+1}{\kappa\nu_i+1},
\]
and taking minima gives
$p^*(\kappa)\le\min_i(a_i+1)/(\kappa\nu_i+1)$.

\medskip
\noindent\emph{Step 3: Proof of part~\textup{(iii)}.}
For each $i$,
\[
\frac{a_i+1}{(\kappa+1)\nu_i-M_i}
\;\ge\;
\frac{a_i+1}{(\kappa+1)\nu_i},
\]
with equality if and only if $M_i=0$. Hence
$p^*(\kappa)=\min_i(a_i+1)/\bigl((\kappa+1)\nu_i-M_i\bigr)$ equals
$\overline{p}(\kappa)=\min_i(a_i+1)/\bigl((\kappa+1)\nu_i\bigr)$ if and
only if the latter minimum is attained at some divisor $E_j$ with $M_j=0$;
by Lemma~\ref{lem:chain_rule}, such a divisor is necessarily a reduced
strict-transform component along which $\{f=0\}$ is generically smooth.

If $\nabla f(0)\neq0$, then after replacing $f$ by its reduced germ the zero
set is a smooth hypersurface; its strict transform $E_0$ has
$\nu_{E_0}=1$, $M_{E_0}=0$, $a_{E_0}=0$, so $E_0$ realizes
$\mathrm{rlct}_0(f)=1$ (Remark~\ref{rem:rlct_divisorial}) with $M_{E_0}=0$,
giving $p^*(\kappa)=\overline{p}(\kappa)=1/(\kappa+1)$.

Conversely, suppose the minimum defining $\mathrm{rlct}_0(f)$ is attained
only by divisors with $M_i\ge1$. Then for each such divisor the inequality
above is strict, while every other divisor has
$(a_i+1)/\bigl((\kappa+1)\nu_i-M_i\bigr)
\ge(a_i+1)/\bigl((\kappa+1)\nu_i\bigr)
>\overline{p}(\kappa)$; hence $p^*(\kappa)>\overline{p}(\kappa)$.
Finally, if $f$ has an isolated zero at the origin and $n\ge2$, then
necessarily $\nabla f(0)=0$ (otherwise the zero set would be a hypersurface
through $0$), and every real component $E_i$ of the total transform
contributing to the minima defining $p^*(\kappa)$ and $\mathrm{rlct}_0(f)$
(Remark~\ref{rem:real_resolution}) satisfies $\pi(E_i)=\{0\}$, so
$M_i\ge1$ for all such $i$ by Lemma~\ref{lem:chain_rule}; the strict inequality
follows. This completes the proof.
\end{proof}

\begin{remark}\label{rem:RLCT_role}
Theorem~\ref{thm:admissibility} shows that the real log--canonical threshold
provides a computable sufficient condition for the exact
$L^p_{\mathrm{loc}}$ threshold of $\nabla K$ (and, via
Corollary~\ref{cor:morrey_bridge}\,(ii), a computable sufficient condition
for Morrey admissibility), while the exact threshold is determined by the
finer resolution data. In particular, the
bound $p<\overline{p}(\kappa)=\mathrm{rlct}_0(f)/(\kappa+1)$ is generally
not sharp, but it captures the dominant geometric contribution to local
integrability. The RLCT bound is sharp precisely when the divisors realizing
$\mathrm{rlct}_0(f)$ include a reduced strict-transform component along which
$\{f=0\}$ is generically smooth (equivalently, a divisor with $M_i=0$); this
is the mildly singular regime, typified by smooth $f$. In the genuinely
singular regime---in particular whenever $f$ has an isolated zero at the
origin in dimension $n\ge2$---the bound is always strict, because the
gradient of $f$ necessarily vanishes along every divisor lying over the
origin. The gap between $\overline{p}(\kappa)$ and $p^*(\kappa)$ is thus a
quantitative measure of the degeneracy of $\nabla f$ along the dominant
divisors, and the sharp threshold is genuinely divisorial rather than an
RLCT phenomenon.
\end{remark}

\begin{proposition}
\label{prop:RLCT_vs_exact}
Let $f$ be a real-analytic function with $f(0)=0$, $f\not\equiv0$, let
$\kappa>0$, and let $(\nu_i,M_i,a_i)$ be the numerical data associated with
a log-resolution of $(f,\mathcal J_f)$ as in
Theorem~\ref{thm:admissibility}. Then
$\overline{p}(\kappa)\le p^*(\kappa)$, with equality if and only if some
divisor realizing $\mathrm{rlct}_0(f)$ satisfies $M_i=0$.
\end{proposition}

\begin{proof}
This is the content of Theorem~\ref{thm:admissibility}\,(ii)--(iii).
\end{proof}

\begin{corollary}
\label{cor:sharpness}
Under the assumptions of Proposition~\ref{prop:RLCT_vs_exact}:

\begin{enumerate}
\item[(i)] If $\nabla f(0)\neq0$, then
$p^*(\kappa)=\overline{p}(\kappa)=1/(\kappa+1)$: the threshold is governed
by the smooth reduced zero set, and the RLCT bound is attained.

\item[(ii)] If $f$ has an isolated zero at the origin and $n\ge2$, then
$\overline{p}(\kappa)<p^*(\kappa)$ strictly. In particular this covers every
germ with isolated real zero already given in coordinates adapted in the
sense of Theorem~\ref{thm:real-newton}: for these, the RLCT bound is
computable via the Newton distance, but it systematically
\emph{underestimates} the exact threshold.

\item[(iii)] The size of the gap is controlled by the gradient orders: for
every divisor $E_j$ realizing $\mathrm{rlct}_0(f)$,
\[
p^*(\kappa)\;\le\;\frac{a_j+1}{(\kappa+1)\nu_j - M_j},
\]
so larger vanishing $M_j$ of the Jacobian ideal along the dominant divisors
yields a larger exact threshold relative to the RLCT prediction.
\end{enumerate}
\end{corollary}

\subsection*{A closed formula in the quasi-homogeneous case}

For quasi-homogeneous singularities the exact threshold admits a closed
formula that requires no explicit resolution, and which we will use
systematically in Section~\ref{sec:examples}. Recall that $f$ is
quasi-homogeneous with weights $w=(w_1,\dots,w_n)\in\mathbb Z_{>0}^n$ and
weighted degree $d$ if
\[
f\bigl(\varepsilon^{w_1}x_1,\dots,\varepsilon^{w_n}x_n\bigr)
=\varepsilon^{d} f(x)
\qquad\text{for all }\varepsilon>0 .
\]
We write $|w|=w_1+\cdots+w_n$ and $w_{\max}=\max_i w_i$, and denote by
$Z_{\mathbb R}(f)$ the real zero set of $f$ near the origin.

\begin{proposition}[Quasi-homogeneous threshold]
\label{prop:quasihomog}
Let $f$ be a quasi-homogeneous polynomial with weights $w$, degree $d$, an
isolated critical point at the origin, and depending on every variable, and
let $\kappa>0$, $K=|f|^{-\kappa}$. Then:
\begin{enumerate}
\item[(i)] if $Z_{\mathbb R}(f)=\{0\}$,
\[
p^*(\kappa)=\frac{|w|}{\kappa d + w_{\max}};
\]
\item[(ii)] if $Z_{\mathbb R}(f)\supsetneq\{0\}$,
\[
p^*(\kappa)=\min\Bigl(\frac{|w|}{\kappa d + w_{\max}},\,
\frac{1}{\kappa+1}\Bigr),
\]
the second value being the contribution of the smooth part of the zero set.
\end{enumerate}
In either case,
$\mathrm{rlct}_0(f)\le|w|/d$, so
$\overline{p}(\kappa)\le(|w|/d)/(\kappa+1)$, and the RLCT bound is strict
in case \textup{(i)} and in case \textup{(ii)} whenever
$|w|/(\kappa d+w_{\max})<1/(\kappa+1)$.
\end{proposition}

\begin{proof}
Write $\delta_\varepsilon(x)=(\varepsilon^{w_1}x_1,\dots,\varepsilon^{w_n}x_n)$
and let $\rho$ be a continuous weighted quasi-norm with
$\rho(\delta_\varepsilon x)=\varepsilon\rho(x)$. Decompose a punctured
neighborhood of the origin into the dyadic annuli
$A_k=\delta_{2^{-k}}(A_0)$, $A_0=\{1/2\le\rho\le1\}$, and change variables:
\[
\int_{A_k}|\nabla K|^p\,dx
=
2^{-k|w|}\int_{A_0}\bigl|\nabla K(\delta_{2^{-k}}x)\bigr|^p\,dx .
\]
Each partial derivative $\partial_j f$ is quasi-homogeneous of degree
$d-w_j$, so
\[
\bigl|\partial_j K(\delta_\varepsilon x)\bigr|
=\kappa\,\varepsilon^{d-w_j}\,|\partial_j f(x)|\cdot
\varepsilon^{-(\kappa+1)d}|f(x)|^{-(\kappa+1)} .
\]
Since $w_j\le w_{\max}$ for every $j$ and $0<\varepsilon<1$, we have
$\varepsilon^{2(d-w_j)}\le\varepsilon^{2(d-w_{\max})}$ for every $j$, so
summing in quadrature over $j$ gives the \emph{upper} bound, uniform on
$A_0$,
\begin{equation}
\label{eq:qh_upper}
\bigl|\nabla K(\delta_\varepsilon x)\bigr|
\;\le\;
C\,\varepsilon^{-(\kappa d + w_{\max})}\,
\frac{|\nabla f(x)|}{|f(x)|^{\kappa+1}}.
\end{equation}
Consequently,
\[
\int_{A_k}|\nabla K|^p\,dx
\;\le\;
C^p\,2^{-k\bigl(|w|-p(\kappa d+w_{\max})\bigr)}\,I(p),
\qquad
I(p):=\int_{A_0}\frac{|\nabla f(x)|^{p}}{|f(x)|^{(\kappa+1)p}}\,dx .
\]

\smallskip
\noindent\emph{Case (i): $Z_{\mathbb R}(f)=\{0\}$.}
Here $f$ does not vanish on the compact set $A_0$, so $I(p)<\infty$ for
every $p$, and if $p(\kappa d+w_{\max})<|w|$ the geometric series in $k$
converges, giving $\nabla K\in L^p_{\mathrm{loc}}$. For the converse, fix an
index $j^*$ with $w_{j^*}=w_{\max}$. We seek a point $x_1\in A_0$ at which
\emph{both} $f(x_1)\ne0$ \emph{and} $\partial_{j^*}f(x_1)\ne0$ hold
simultaneously, so that the lower bound below is finite and nonzero at
once. The first condition holds at every point of $A_0$, since
$Z_{\mathbb R}(f)=\{0\}\notin A_0$. For the second: since $f$ depends on
every variable by hypothesis, $\partial_{j^*}f$ is a nonzero real-analytic
(in fact polynomial) function, so its zero set has empty interior in
$\mathbb R^n$; as $A_0$ has nonempty interior, there is a point
$x_1\in\mathrm{int}(A_0)$ with $\partial_{j^*}f(x_1)\neq0$, and this $x_1$
automatically satisfies $f(x_1)\ne0$ as noted. By continuity, shrinking to a ball
$B_\delta(x_1)\subset A_0$, we have $|\partial_{j^*}f(x)|\ge c_1>0$ on
$B_\delta(x_1)$, while $0<|f(x)|\le c_2<\infty$ throughout the compact set
$A_0$. Hence, keeping only the $j^*$-th term in the sum of squares defining
$|\nabla K(\delta_\varepsilon x)|$,
\[
|\nabla K(\delta_\varepsilon x)|
\;\ge\;
\kappa\,\varepsilon^{-(\kappa+1)d}\,c_2^{-(\kappa+1)}\,
\varepsilon^{d-w_{\max}}\,c_1
\;=\;
c_3\,\varepsilon^{-(\kappa d+w_{\max})}
\qquad\text{on }B_\delta(x_1),
\]
for a constant $c_3>0$ and all small $\varepsilon$. Consequently
\[
\int_{A_k}|\nabla K|^p\,dx
\;\ge\;
c_3^p\,|B_\delta(x_1)|\,2^{-k(|w|-p(\kappa d+w_{\max}))},
\]
and if $p(\kappa d+w_{\max})\ge|w|$ this lower bound does not tend to $0$,
so the series $\sum_k\int_{A_k}|\nabla K|^p\,dx$ diverges and
$\nabla K\notin L^p_{\mathrm{loc}}$. This proves (i).

\smallskip
\noindent\emph{Case (ii): $Z_{\mathbb R}(f)\supsetneq\{0\}$.}
The upper bound \eqref{eq:qh_upper} again gives $\nabla K\in L^p_{\mathrm{loc}}$
whenever $p(\kappa d+w_{\max})<|w|$ \emph{and} $I(p)<\infty$. We now locate
the finiteness threshold for $I(p)$ directly, by a co-area argument that
uses only the existence of a single suitable point, avoiding any genericity
assumption on a distinguished coordinate.

Since the origin is an isolated critical point of $f$, we have
$\nabla f\neq0$ at every point of $Z_{\mathbb R}(f)\cap A_0$; fix any point
$x_0$ in this nonempty compact set. By the implicit function theorem, there
is a ball $B_\delta(x_0)\subset A_0$ on which $Z_{\mathbb R}(f)$ is a smooth
hypersurface and, shrinking $\delta$ if necessary so that $|\nabla f|$ stays
within a factor of $2$ of $|\nabla f(x_0)|\neq0$, we obtain the two-sided bounds
\[
c_1\,\mathrm{dist}\bigl(x,Z_{\mathbb R}(f)\bigr)
\;\le\;
|f(x)|
\;\le\;
c_2\,\mathrm{dist}\bigl(x,Z_{\mathbb R}(f)\bigr),
\qquad x\in B_\delta(x_0),
\]
\[
c_3\;\le\;|\nabla f(x)|\;\le\;c_4,
\qquad x\in B_\delta(x_0),
\]
for constants $0<c_1\le c_2$, $0<c_3\le c_4$. Foliating $B_\delta(x_0)$ by
level sets of $x\mapsto\mathrm{dist}(x,Z_{\mathbb R}(f))$ (smooth near $x_0$
since $\nabla f(x_0)\neq0$), the co-area formula gives
\[
\int_{B_\delta(x_0)}\frac{|\nabla f(x)|^p}{|f(x)|^{(\kappa+1)p}}\,dx
\;\asymp\;
\int_0^{\delta'} t^{-(\kappa+1)p}\,dt,
\]
which is finite if and only if $(\kappa+1)p<1$, and infinite (with the
partial integrals over $B_\delta(x_0)\setminus B_{2^{-k}\delta'}(Z_{\mathbb R}(f))$
diverging as $k\to\infty$) if $(\kappa+1)p\ge1$. Since
$B_\delta(x_0)\subset A_0$ and the integrand is nonnegative,
$I(p)\ge\int_{B_\delta(x_0)}(\cdots)\,dx$, so $I(p)=\infty$ whenever
$(\kappa+1)p\ge1$; combined with the elementary bound
$I(p)\le\sup_{A_0\setminus B_{\delta/2}(x_0)}(\cdots)\cdot|A_0|+\int_{B_\delta(x_0)}(\cdots)$
(the first term finite since $f$ is bounded away from $0$ away from
$Z_{\mathbb R}(f)$, and $Z_{\mathbb R}(f)\cap A_0$ is covered by finitely
many such balls by compactness), we obtain $I(p)<\infty$ if and only if
$(\kappa+1)p<1$.

When $(\kappa+1)p<1$, $I(p)<\infty$ and the upper bound above gives
$\nabla K\in L^p_{\mathrm{loc}}$ provided also $p(\kappa d+w_{\max})<|w|$.
When $(\kappa+1)p\ge1$, the divergence of $I(p)$ already forces
$\int_{A_0}|\nabla K|^p\,dx=\infty$ (taking $k=0$ in the annulus
decomposition, since $|\nabla K|\gtrsim|\nabla f|/|f|^{\kappa+1}$ pointwise
away from where $\nabla K$ might vanish, which does not happen near
$Z_{\mathbb R}(f)$), so $\nabla K\notin L^p_{\mathrm{loc}}$. Symmetrically,
when $p(\kappa d+w_{\max})\ge|w|$, the divergence argument of case (i)
applies verbatim: since $Z_{\mathbb R}(f)\cap A_0$ and the zero set of the
nonzero polynomial $\partial_{j^*}f$ ($w_{j^*}=w_{\max}$) both have empty
interior in $A_0$, their union does too, so there is a point
$x_1\in\mathrm{int}(A_0)$ outside both, at which the same estimate as in
case (i) shows $\int_{A_k}|\nabla K|^p\,dx\gtrsim2^{-k(|w|-p(\kappa d+w_{\max}))}$,
giving divergence and $\nabla K\notin L^p_{\mathrm{loc}}$. Combining,
$\nabla K\in L^p_{\mathrm{loc}}$
if and only if both $p(\kappa d+w_{\max})<|w|$ and $(\kappa+1)p<1$, i.e.\
$p<\min\bigl(|w|/(\kappa d+w_{\max}),\,1/(\kappa+1)\bigr)$. This proves (ii).

\smallskip
Finally, $\mathrm{rlct}_0(f)\le|w|/d$ follows from
$\int_{A_k}|f|^{-c}\asymp2^{-k(|w|-cd)}$. For the strictness claims, note
first that $w_{\max}<d$: since $0$ is a critical point, $f$ contains no
linear monomial, so every monomial involving the variable of maximal weight
has weighted degree strictly greater than $w_{\max}$. Hence
\[
\frac{|w|}{\kappa d+w_{\max}}
\;>\;
\frac{|w|}{(\kappa+1)d}
\;\ge\;
\frac{\mathrm{rlct}_0(f)}{\kappa+1}
=\overline{p}(\kappa),
\]
which gives strictness in case (i), and in case (ii) whenever the
exceptional value is the minimum.
\end{proof}

\begin{remark}
In resolution-theoretic terms, the two values in
Proposition~\ref{prop:quasihomog} are exactly the divisorial ratios of
Theorem~\ref{thm:admissibility}: the weighted blow-up with weights $w$
produces an exceptional divisor $E$ with
$\nu_E=d$, $a_E=|w|-1$, and $M_E=d-w_{\max}$ (each $\partial_j f$ being
quasi-homogeneous of degree $d-w_j$), so that
\[
\frac{a_E+1}{(\kappa+1)\nu_E-M_E}
=\frac{|w|}{\kappa d + w_{\max}},
\]
while a reduced smooth branch of $Z_{\mathbb R}(f)$ contributes
$\nu=1$, $M=0$, $a=0$, hence $1/(\kappa+1)$. The direct dyadic proof above
avoids the bookkeeping of the auxiliary divisors that arise when resolving
the quotient singularities of the weighted blow-up.
\end{remark}

\subsection*{A computational criterion for Morrey admissibility}

We now explain why Newton--polyhedron methods provide an effective computational
tool for determining admissible singularity regimes.
In real dimension two, a fundamental result originating in the work of Varchenko
\cite{varchenko1976newton} and refined by Collins \cite{collins2017real} shows
that, after passing to adapted coordinates, the real Newton distance
determines the real log--canonical threshold under the hypotheses of
Theorem~\ref{thm:real-newton} below.

\begin{definition}[Newton polyhedron]
Let \(f(x,y)=\sum_{(i,j)\in\mathbb{N}^2} a_{ij} x^i y^j\) be a real--analytic germ at
\((0,0)\).
The \emph{Newton polyhedron} \(\NP(f)\) is the convex hull of the set
\[
\bigcup_{\substack{(i,j)\in\mathbb{N}^2\\ a_{ij}\neq 0}}
\bigl((i,j)+\mathbb{R}_{\ge0}^2\bigr)
\subset \mathbb{R}_{\ge0}^2.
\]
\end{definition}

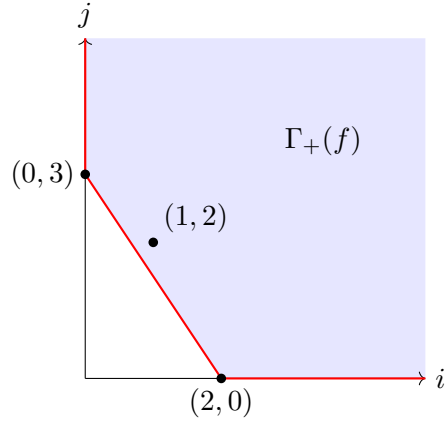
\begin{figure}[h]
\centering
\begin{tikzpicture}[scale=0.9]
\begin{scope}
    \clip (0,0) rectangle (5,5);
    \fill[blue!10] (0,3) -- (2,0) -- (5,0) -- (5,5) -- (0,5) -- cycle;
\end{scope}

\draw[->] (0,0) -- (5,0) node[right] {$i$};
\draw[->] (0,0) -- (0,5) node[above] {$j$};

\draw[thick, red] (0,5) -- (0,3) -- (2,0) -- (5,0);

\fill (2,0) circle (2pt) node[below] {$(2,0)$};
\fill (1,2) circle (2pt) node[above right] {$(1,2)$};
\fill (0,3) circle (2pt) node[left] {$(0,3)$};

\node at (3.5,3.5) {$\NP(f)$};
\end{tikzpicture}
\caption{Newton polyhedron of \(f(x,y)=x^2+xy^2+y^3\). The monomial $(1,2)$
lies strictly in the interior of $\NP(f)$ (since $\tfrac12+\tfrac23>1$) and
does not contribute a vertex: the compact face is the segment joining
$(2,0)$ and $(0,3)$.}
\end{figure}

\begin{definition}[Newton distance]
The \emph{Newton distance} of \(f\) is defined as
\[
d_{\mathrm{NP}}(f)
:= \inf\{\, t>0:\ (t,t)\in\NP(f)\,\}.
\]
Equivalently, it is the first intersection of the diagonal
\(\{(t,t)\}\) with the Newton polyhedron.
\end{definition}

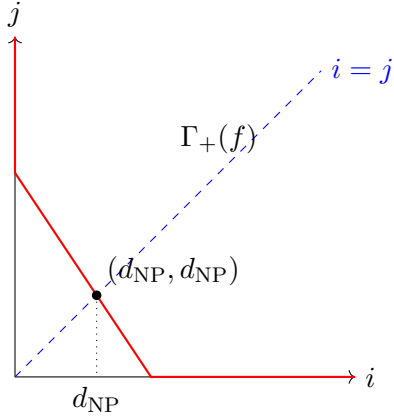
\begin{figure}[h]
\centering
\begin{tikzpicture}[scale=0.9]
\draw[->] (0,0) -- (5,0) node[right] {$i$};
\draw[->] (0,0) -- (0,5) node[above] {$j$};

\draw[thick, red] (0,5) -- (0,3) -- (2,0) -- (5,0);

\draw[dashed, blue] (0,0) -- (4.5,4.5) node[right] {$i=j$};

\coordinate (D) at (1.2, 1.2);
\fill[black] (D) circle (2pt);
\draw[dotted] (D) -- (1.2,0) node[below] {$d_{\mathrm{NP}}$};

\node[above right] at (D) {$(d_{\mathrm{NP}},d_{\mathrm{NP}})$};
\node at (3,3.5) {$\NP(f)$};
\end{tikzpicture}
\caption{Newton distance as the intersection of the diagonal with \(\NP(f)\);
for \(f=x^2+xy^2+y^3\) the compact face is the segment from $(0,3)$ to
$(2,0)$ and $d_{\mathrm{NP}}(f)=6/5$.}
\end{figure}
\begin{theorem}[Collins {\cite[Thm.~1.4]{collins2017real}}]\label{thm:real-newton}

Let \(f\) be a real--analytic germ at \(0\in\mathbb{R}^2\), and let
$\delta_{\mathrm{NP}}(f)^{-1}=\inf\{t>0:(t,t)\in\Gamma_+(f)\}$ be its Newton
distance (computed in the given coordinates). There exists a
real--analytic change of coordinates
\[
(\tilde{x},\tilde{y})=(x-Q(y),y)
\quad \text{or}\quad
(\tilde{x},\tilde{y})=(x,y-Q(x)),
\]
with \(Q\) a suitable real-analytic function, such that, in the new
coordinates,
\[
\operatorname{rlct}_0(f)=\frac{1}{d_{\mathrm{NP}}(f)},
\]
where $d_{\mathrm{NP}}(f)=\delta_{\mathrm{NP}}(f)^{-1}$ is the Newton
distance recomputed in the adapted coordinates. (In Collins' normalization
for the complex case, an extra factor of $2$ appears; the real case, which
is the one used throughout this paper, carries no such factor.) The
adapted coordinates are characterized precisely by Collins'
\cite[Prop.~3.10]{collins2017real}: writing the roots of (a Weierstrass
factorization of) $f$ as Puiseux series clustered by leading exponent, the
coordinates work whenever each cluster of roots sharing a given leading
exponent has cardinality at most $\delta_{\mathrm{NP}}(f)$; when this fails
for the face of $\Gamma_+(f)$ meeting the diagonal (the ``main face''),
Collins' change of variables explicitly removes the offending cluster,
strictly decreasing the Newton distance until the condition holds.
\end{theorem}

\begin{remark}[Adapted coordinates matter]
\label{rem:adapted_coords}
The formula $\mathrm{rlct}_0(f)=1/d_{\mathrm{NP}}(f)$ is valid only in the
adapted coordinates produced by Theorem~\ref{thm:real-newton}, in which
$d_{\mathrm{NP}}(f)\ge1$ (consistently with the fact that
$\mathrm{rlct}_0(f)\le1$ for plane germs vanishing at the origin: a germ of
order $k\ge2$ has $\mathrm{rlct}_0\le2/k\le1$, and a germ of order $1$ has a
smooth reduced zero curve, hence $\mathrm{rlct}_0=1$ by
Remark~\ref{rem:rlct_divisorial}). Applying it in non-adapted coordinates
can overestimate the RLCT: for instance, $f=x+y^n$ has
$d_{\mathrm{NP}}=n/(n+1)<1$ in the given coordinates, but the substitution
$\tilde x=x+y^n$ shows $f$ is smooth with $\mathrm{rlct}_0(f)=1$, not
$(n+1)/n$.
\end{remark}

From a practical perspective, Theorem~\ref{thm:admissibility}, combined with
Proposition~\ref{prop:quasihomog}, offers a direct and computable criterion
for ascertaining the critical Morrey admissibility regime of a singular
interaction potential. Instead of using ad hoc integral estimates, a finite
geometric method can be used to decide whether kernels of the form
$K=|f|^{-\kappa}$ are admissible.

More specifically, for a real-analytic function $f$ satisfying
Hypothesis~\ref{hyp:dominant}, the admissible Morrey exponent is obtained
through the following steps:

\begin{enumerate}
    \item Construct the Newton polyhedron $\Gamma_+(f)$ from the Taylor
    expansion of $f$ at the origin, in adapted coordinates
    (Theorem~\ref{thm:real-newton}).

    \item Find the Newton distance $d_{\mathrm{NP}}(f)$, the intersection
    point of the diagonal with the boundary of $\Gamma_+(f)$. In these
    adapted coordinates (Theorem~\ref{thm:real-newton}), the real
    log-canonical threshold satisfies
    $\mathrm{rlct}_0(f)=1/d_{\mathrm{NP}}(f)$, and
    Theorem~\ref{thm:admissibility}\,(ii) yields the computable sufficient
    condition
    \[
    p \;<\; \overline{p}(\kappa)
    \;=\;\frac{1}{(\kappa+1)\,d_{\mathrm{NP}}(f)}
    \;\Longrightarrow\;
    \nabla K\in L^p_{\mathrm{loc}} .
    \]

    \item If $f$ is quasi-homogeneous (as in all examples of
    Section~\ref{sec:examples}), the \emph{exact} threshold is given in
    closed form by Proposition~\ref{prop:quasihomog}:
    \[
    p^*(\kappa)
    =\frac{|w|}{\kappa d + w_{\max}}
    \quad\text{if } Z_{\mathbb R}(f)=\{0\},
    \]
    \[
    p^*(\kappa)
    =\min\Bigl(\frac{|w|}{\kappa d + w_{\max}},\frac{1}{\kappa+1}\Bigr)
    \quad\text{otherwise,}
    \]
    and $\nabla K \in L^p_{\mathrm{loc}}$ if and only if $p < p^*(\kappa)$.

    \item In the general case, the exact threshold requires, in addition to
    $(\nu_i,a_i)$, the gradient orders $M_i=\nu_{E_i}(\mathcal J_f)$ from a
    log-resolution of $(f,\mathcal J_f)$, via
    Theorem~\ref{thm:admissibility}\,(i).
\end{enumerate}

This process gives a clear characterisation of the maximal integrability of
$\nabla K$: a resolution-free sufficient condition from Newton geometry, and
the exact condition---the one needed as local integrability input for the
fixed-point argument in the Morrey framework---from the divisorial data of
$(f,\mathcal J_f)$.

\subsection{The bridge to Morrey admissibility}
\label{sec:morrey_bridge}

Theorem~\ref{thm:admissibility} is a statement about local
$L^p_{\mathrm{loc}}(\mathbb R^n)$-integrability of $\nabla K$; it does not,
by itself, refer to the well-posedness theorem of
\cite{suleiman2020aggregation,suleiman2023existence}. We now make this
connection precise, and we do so by reading off the exact hypothesis of
\cite[Thm.~3.1]{suleiman2023existence} rather than by paraphrasing it.

Recall from Section~\ref{sec:morrey_framework} that
\cite[Thm.~3.1]{suleiman2023existence} requires $n\ge2$, $0\le\lambda<n-1$,
exponents $1<p<q<\infty$ with $p=n-\lambda$, $1<q/2$, $1/p+1/q<1$, initial
data $u_0\in M_{p,\lambda}$, and---crucially, as the \emph{only} hypothesis
on the kernel---
\[
\nabla K\in L^1(\mathbb R^n).
\]
This is a statement about the \emph{global} $L^1$ norm of $\nabla K$; the
Morrey exponent $p=n-\lambda$ that appears alongside it governs only the
scaling class $M_{p,\lambda}$ in which the initial data $u_0$ is posed
(via the invariance $u_0\mapsto\alpha u_0(\alpha\cdot)$ of the rescaled
equation), and plays no role in the integrability requirement on $\nabla K$
itself: the bilinear estimate of Lemma~\ref{lem:bilinear} is proved in
\cite[Lemma~4.3]{suleiman2023existence} using Young's inequality
$\|\nabla K*u\|_{M_{p,\lambda}}\le\|\nabla K\|_{L^1}\|u\|_{M_{p,\lambda}}$
(mechanism (A) of Section~\ref{subsec:rlct_definition}), for \emph{any}
admissible $(p,\lambda)$, once $\|\nabla K\|_{L^1(\mathbb R^n)}<\infty$.
This is Mechanism (A) alone; we do not use Mechanism (B). Combining this
reading of \cite[Thm.~3.1]{suleiman2023existence} with
Theorem~\ref{thm:admissibility} yields the following bridge.

\begin{corollary}[Morrey admissibility criterion]
\label{cor:morrey_bridge}
Let $K$ satisfy Hypothesis~\ref{hyp:dominant} with data $f$, $\kappa$, near
the origin, and assume in addition that $K$ is modified, outside a fixed
compact neighborhood of the origin, to be smooth with $\nabla K$ compactly
supported (or otherwise satisfying $(1-\chi_{B_{\varepsilon_0}})\nabla K\in L^1(\mathbb R^n)$).
Let $p^*(\kappa)$ be the exact threshold of Theorem~\ref{thm:admissibility}.
Then:
\begin{enumerate}
\item[\textup{(i)}] \textbf{Local criterion.}
$\nabla K\in L^1_{\mathrm{loc}}(\mathbb R^n)$ near the origin if and only if
\[
1 \;<\; p^*(\kappa).
\]
\item[\textup{(ii)}] \textbf{Computable RLCT-based sufficient condition.}
\[
1 \;<\; \frac{\mathrm{rlct}_0(f)}{\kappa+1}
\qquad\Longleftrightarrow\qquad
\kappa+1 \;<\; \mathrm{rlct}_0(f)
\]
implies (i).
\item[\textup{(iii)}] \textbf{Application of \cite[Thm.~3.1]{suleiman2023existence}.}
If \textup{(i)} holds, then $\nabla K\in L^1(\mathbb R^n)$ globally, and the
hypothesis of \cite[Thm.~3.1]{suleiman2023existence} is satisfied for
\emph{every} choice of Morrey parameters $n\ge2$, $0\le\lambda<n-1$,
$p=n-\lambda$, $q\in(p,\infty)$ with $1<q/2$ and $1/p+1/q<1$: the
well-posedness, uniqueness, continuous dependence, and (for odd/even data)
symmetry conclusions of that theorem apply to the aggregation
equation~\eqref{eq:agg} with kernel $K$, for small data $u_0\in M_{p,\lambda}$.
\end{enumerate}
\end{corollary}

\begin{proof}
Part (i) is Theorem~\ref{thm:admissibility}(i) applied with $p=1$. Part (ii)
follows from Theorem~\ref{thm:admissibility}(ii): if
$1<\mathrm{rlct}_0(f)/(\kappa+1)=\overline{p}(\kappa)$ then
$1<\overline{p}(\kappa)\le p^*(\kappa)$, giving (i). For part (iii): under
(i), $\nabla K\in L^1$ on a neighborhood of the origin
(Theorem~\ref{thm:admissibility}(i) with $p=1$, together with continuity of
$\nabla K$ away from the origin on the fixed compact neighborhood), and by
the global modification hypothesis, $\nabla K\in L^1(\mathbb R^n)$ overall.
This is exactly the hypothesis of \cite[Thm.~3.1]{suleiman2023existence}
recalled above, for any admissible $(n,\lambda,p,q)$; the theorem's
conclusion applies verbatim, with the bilinear estimate constants $M_1,M_2$
of Lemma~\ref{lem:bilinear} given explicitly in terms of
$\|\nabla K\|_{L^1(\mathbb R^n)}$.
\end{proof}

We emphasize the logical structure: part (i) is fully proved in this paper,
as a corollary of Theorem~\ref{thm:admissibility}; part (iii) invokes
\cite[Thm.~3.1]{suleiman2023existence} as a black box, applied to the
geometric (and global-modification) input supplied by (i). This separates
cleanly what is established here---the exact geometric threshold governing
$L^1_{\mathrm{loc}}$ integrability of $\nabla K$ near its singularity---from
what is imported from the existing PDE theory. In particular,
the criterion $\kappa+1<\mathrm{rlct}_0(f)$, or more sharply $p^*(\kappa)>1$,
is entirely \emph{independent} of the Morrey parameters $(p,\lambda)$: once
it holds, \cite[Thm.~3.1]{suleiman2023existence} applies simultaneously for
the \emph{whole} admissible range $0\le\lambda<n-1$ of the solution space,
because $\lambda$ enters that theorem only through the space in which
$u_0$ is posed, not through any integrability requirement on $K$.

\begin{remark}[Genuinely singular kernels fail Regime A]
\label{rem:regime_C_examples}
Condition (i) of Corollary~\ref{cor:morrey_bridge} is a genuine constraint on
$\kappa$ and $f$, not an automatic consequence of Hypothesis~\ref{hyp:dominant}.
For a mildly singular kernel ($\nabla f(0)\neq0$, so
$p^*(\kappa)=1/(\kappa+1)\le1$ for every $\kappa>0$ by
Theorem~\ref{thm:admissibility}(iii)), condition (i) fails for every
$\kappa>0$: smooth-type singularities of this kind never enter Regime A.
More strikingly, several of the quasi-homogeneous examples of
Section~\ref{sec:examples} have $p^*(\kappa)<1$ already at $\kappa=1$ (e.g.\
the family $f=x_1^4+x_2^6$ of Section~\ref{sec:examples} below, where
$p^*(1)=1/3$), and hence fail Corollary~\ref{cor:morrey_bridge}(i) at that
value of $\kappa$, although---as we verify in
Section~\ref{subsec:morrey_bridge_check} below---the very same divisorial
type $f=x_1^4+x_2^6$ does satisfy it for $\kappa$ sufficiently small. In this
sense the geometric threshold $p^*(\kappa)$, as a decreasing function of
$\kappa$, genuinely restricts which powers $|f|^{-\kappa}$ of a given
singular divisor fall within the scope of
\cite[Thm.~3.1]{suleiman2023existence}; kernels or exponents with
$p^*(\kappa)\le1$ (Regime C of Remark~\ref{rem:three_regimes}) are outside
that scope, and would require a different functional framework---this
restriction has nothing to do with the choice of $\lambda$, which is free
throughout the whole admissible range once (i) holds.
\end{remark}

\subsection{A worked verification: the Newtonian kernel}
\label{subsec:morrey_bridge_check}

We illustrate Corollary~\ref{cor:morrey_bridge} on the classical Newtonian
kernel $K_{\mathrm{Newt}}(x)\sim|x|^{2-d}$ in $\mathbb R^d$ ($n=d$), suitably
truncated to be smooth and compactly supported in gradient outside a fixed
neighborhood of the origin (recall from the discussion after
Hypothesis~\ref{hyp:dominant} that the exact power-law kernel cannot itself
be globally $L^1$-gradient, being homogeneous). We stress that we are
\emph{not} asserting that the untruncated global kernel
$|x|^{2-d}$ satisfies the $L^1(\mathbb R^n)$ hypothesis of
\cite[Thm.~3.1]{suleiman2023existence}---it manifestly does not, by
homogeneity. What we certify is that its local singular germ at the origin
meets the exact divisorial criterion $p^*(\kappa)>1$, and that this local
certification, combined with the mild global truncation constructed in
Corollary~\ref{cor:morrey_bridge}, produces a kernel to which
\cite[Thm.~3.1]{suleiman2023existence} literally applies. Locally near the origin,
$K_{\mathrm{Newt}}$ is recovered in Section~\ref{subsec:newtonian} as
$f=|x|^2$, $\kappa=(d-2)/2$, with
\[
p^*(\kappa)=\frac{d}{d-1},
\qquad
\mathrm{rlct}_0(f)=\frac d2,
\qquad
\overline{p}(\kappa)=\frac{\mathrm{rlct}_0(f)}{\kappa+1}=1
\]
(from $\nu_E=2$, $a_E=d-1$, $M_E=1$, and $\kappa+1=d/2$).

\emph{RLCT-only sufficient condition (ii).} This reads $1<\overline{p}(\kappa)=1$,
which fails (with equality, not strict inequality): the RLCT bound alone is
borderline and does \emph{not} certify Regime A for the Newtonian kernel.

\emph{Exact threshold (i).} Since $d\ge3$ gives $p^*(\kappa)=d/(d-1)\in(1,3/2]$,
condition (i) of Corollary~\ref{cor:morrey_bridge} holds strictly for every
$d\ge3$. By part (iii), \cite[Thm.~3.1]{suleiman2023existence} therefore
applies to the (truncated) Newtonian kernel for \emph{every} admissible
choice of $\lambda\in[0,n-1)$ and compatible $q$---there is no delicate
window in $\lambda$ to locate, because the Morrey parameters play no role in
condition (i) at all. This is exactly the phenomenon anticipated in
Theorem~\ref{thm:admissibility}(iii) and Remark~\ref{rem:RLCT_role}: the
exact divisorial threshold, not the RLCT (which here only reaches the
borderline value $\overline{p}=1$), is what certifies the classical theory.

\emph{A genuinely anisotropic example.} For contrast, consider
$f=x_1^4+x_2^6$ in $\mathbb R^2$, for which Section~\ref{sec:examples}
computes $p^*(\kappa)=5/(12\kappa+3)$ (isolated real zero, so no
strict-transform contribution). Condition (i) of Corollary~\ref{cor:morrey_bridge}
requires $1<5/(12\kappa+3)$, i.e.\
\[
\kappa \;<\; \frac16.
\]
For $\kappa=1$ (the value used for the table in Section~\ref{sec:examples}),
this fails, confirming the failure recorded in
Remark~\ref{rem:regime_C_examples}; for, say, $\kappa=1/10<1/6$, one has
$p^*(1/10)=5/4.2\approx1.19>1$, so Regime A holds and
\cite[Thm.~3.1]{suleiman2023existence} applies to the (truncated) kernel
$|f|^{-1/10}$ for every admissible $\lambda\in[0,n-1)$. This illustrates
concretely that $p^*(\kappa)$ determines a threshold exponent $\kappa$ below
which a given power of a singular divisor enters the scope of
\cite[Thm.~3.1]{suleiman2023existence}, entirely independently of the
Morrey scaling parameter $\lambda$ of the solution space.

Methods from resolution of singularities have long been used to analyze and
control analytic behavior. A classical example is Varchenko's work on Newton
polyhedra \cite{varchenko1976newton}, which shows how resolution data can be
used to describe decay properties of oscillatory integrals. This perspective
was later refined by Greenblatt \cite{greenblatt2005sharp} and
Kamimoto--Nose \cite{kamimoto2007toric}, among others. Related ideas also
appear in motivic integration \cite{denef1998motivic} and in the study of
zeta functions \cite{phong2012algebraic}, where resolution data again plays a
central role.

The log canonical threshold arises naturally in several areas, including
birational geometry \cite{defernex2010bounds}, the theory of plurisubharmonic
functions \cite{demailly2001semicontinuity}, and the study of multiplier ideals
\cite{blickle2004informal,ein2004jumping}. In particular, Musta\c{t}\u{a}'s
valuative characterization \cite{mustata2002singularities}, together with the
asymptotic invariants introduced by Jonsson--Musta\c{t}\u{a}
\cite{jonsson2012valuations}, provides a bridge between these geometric ideas
and the analytic integrability conditions considered here.

More recently, Collins--Greenleaf--Pramanik \cite{collins2016multidimensional}
developed multidimensional resolution techniques, while Phong--Stein--Sturm
\cite{phong1999growth} studied growth properties of real-analytic functions.
Our approach follows this general geometric perspective, adapted to the setting
of nonlocal evolution equations.

\section{Examples: a proof of concept for the admissibility criterion}
\label{sec:examples}

A key advantage of the resolution--based admissibility criterion developed in
Section~\ref{sec:rlct_morrey} is its sensitivity to the anisotropic geometry of
the interaction potential. While many classical kernels studied in the PDE
literature are isotropic and admit radial majorants of the form
$K(x)\sim |x|^{-\alpha}$, aggregation models arising in heterogeneous or
structured environments naturally lead to interaction potentials with strongly
anisotropic singularities.

From a purely analytic viewpoint, such anisotropy is difficult to capture using
isotropic Morrey or Hardy--Littlewood--Sobolev estimates. From a geometric
viewpoint, however, anisotropic concentration phenomena are precisely those
encoded by the Newton polyhedron and, more generally, by the resolution data
of the defining function $f$.

Theorem~\ref{thm:admissibility} shows that the exact $L^p_{\mathrm{loc}}$
threshold of $\nabla K$---and hence, via Corollary~\ref{cor:morrey_bridge},
Morrey admissibility---is organized by resolution--theoretic data associated
with the singularity of the interaction
kernel, rather than by isotropic upper bounds, so that directional
interaction kernels are treated by the same criterion as isotropic ones.

\medskip
We now present a collection of examples illustrating
Theorem~\ref{thm:admissibility} for Morrey
admissibility. Rather than testing kernels against a range of Morrey norms,
each example is treated as a case study in which the Newton polyhedron identifies
the dominant resolution data and, consequently, the admissible singularity regime:

\begin{quote}
\emph{Once the weights and the weighted degree of the dominant quasi-homogeneous
part of $f$ are identified, the admissible integrability range for $\nabla K$
is determined by explicit resolution--theoretic quantities: the exceptional
contribution $|w|/(\kappa d+w_{\max})$ of Proposition~\ref{prop:quasihomog},
together with the contribution $1/(\kappa+1)$ of the smooth part of the real
zero set whenever the latter is not reduced to the origin. The Newton
polyhedron computes the RLCT and hence the computable lower bound
$\overline{p}(\kappa)$, which in all the singular examples below is strictly
smaller than the exact threshold.}
\end{quote}

\begin{remark}[Scope relative to Regime A]
\label{rem:section3_scope}
We record here, once and for all, a point already anticipated in
Remark~\ref{rem:regime_C_examples}: at $\kappa=1$, most of the exact
thresholds $p^*(1)$ computed below are smaller than $1$ (see the tables in
Sections~\ref{subsec:model_family}--\ref{subsec:mixed_example}), and hence,
by Corollary~\ref{cor:morrey_bridge}, the corresponding kernels $K=1/f$ do
\emph{not} satisfy the global $L^1(\mathbb R^n)$ hypothesis of
\cite[Thm.~3.1]{suleiman2023existence} at $\kappa=1$---and this failure has
nothing to do with the choice of Morrey parameters $\lambda$, since
condition (i) of Corollary~\ref{cor:morrey_bridge} does not involve
$\lambda$ at all. This is not a defect of the examples: they are included to
exhibit the geometric mechanism---the interplay of $\nu_i$, $M_i$, $a_i$,
and the Newton polyhedron---in its purest anisotropic form, and, as shown in
Section~\ref{subsec:morrey_bridge_check}, the same families do enter Regime
A once $\kappa$ is taken small enough (equivalently, for a sufficiently
mild power $|f|^{-\kappa}$ of the same singular divisor). The one example
below with a built-in physical normalization---the classical Newtonian
kernel of Section~\ref{subsec:newtonian}, where $\kappa$ is fixed by the
dimension via $\kappa=(d-2)/2$---does satisfy Regime A for every $d\ge3$,
and hence \cite[Thm.~3.1]{suleiman2023existence} applies to it
\emph{simultaneously for the entire admissible range} $0\le\lambda<n-1$, as
verified in Section~\ref{subsec:morrey_bridge_check}.
\end{remark}

Throughout this section we take $\kappa=1$, i.e.\ $K=1/f$, unless stated
otherwise; the general-$\kappa$ formulas follow from
Proposition~\ref{prop:quasihomog} by the same data.

\subsection{A model family: $f(x,y)=x^m+y^n$}
\label{subsec:model_family}

We begin with the prototypical family
\[
f(x,y)=x^m+y^n, \qquad m,n\ge2,
\]
which already exhibits a wide range of anisotropic behaviors.

The Newton polyhedron $\NP(f)$ has vertices at $(m,0)$ and $(0,n)$.
The dominant face is the line segment connecting these two points.
Intersecting this face with the diagonal $(t,t)$ yields
\[
\frac{t}{m}+\frac{t}{n}=1,
\qquad\text{hence}\qquad
t=\frac{mn}{m+n}.
\]
Therefore the Newton distance is
\[
d_{\mathrm{NP}}(f)=\frac{mn}{m+n}.
\]

\begin{figure}[h]
\centering
\begin{tikzpicture}[scale=1.2]
    \begin{scope}
        \clip (0,0) rectangle (4.5,4.5);
        \fill[blue!10] (0,3.5) -- (2.5,0) -- (5,0) -- (5,5) -- (0,5) -- cycle;
    \end{scope}

    \draw[->] (0,0) -- (4.5,0) node[right] {$i$};
    \draw[->] (0,0) -- (0,4.5) node[above] {$j$};

    \draw[dashed, gray] (0,0) -- (4,4) node[right, black] {$i=j$};

    \draw[thick, red] (0,4.5) -- (0,3.5) node[left, black] {$n$}
          -- (2.5,0) node[below, black] {$m$} -- (4.5,0);

    \fill (0,3.5) circle (1.5pt);
    \fill (2.5,0) circle (1.5pt);

    \coordinate (T) at (1.458, 1.458);
    \fill[black] (T) circle (2pt);
    \draw[dotted] (T) -- (1.458,0) node[below] {$d_{\mathrm{NP}}$};
    \draw[dotted] (T) -- (0,1.458) node[left] {$d_{\mathrm{NP}}$};

    \node at (3,3) {$\NP(x^m+y^n)$};
\end{tikzpicture}
\caption{Newton polyhedron and distance for the family $x^m+y^n$.}
\end{figure}
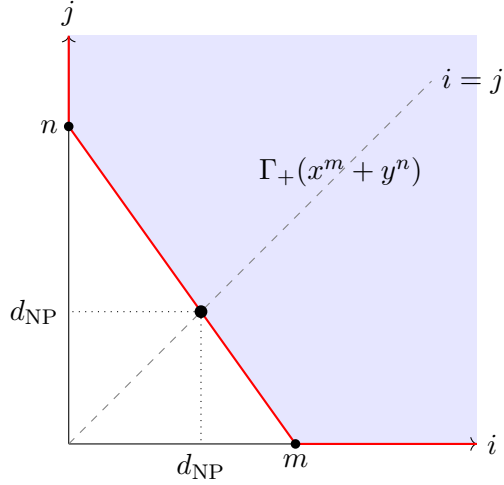

In real dimension two, Theorem~\ref{thm:real-newton} implies that
\[
\mathrm{rlct}_0(f)=\frac{1}{d_{\mathrm{NP}}(f)}
=\frac{m+n}{mn},
\]
(the coordinates are already adapted since $d_{\mathrm{NP}}\ge1$ for
$m,n\ge2$), giving the computable RLCT lower bound
$\overline{p}=\tfrac12\,\mathrm{rlct}_0(f)=\tfrac{m+n}{2mn}$.

\begin{remark}[Independent check]
This value of $\mathrm{rlct}_0$ can be cross-checked directly against
Collins' own computation in the proof of his Theorem~1.3
\cite[Sect.~4]{collins2017real}, where---in the complex normalization,
which carries an extra factor of $2$---he computes, for $f_n=y^n-x^m$ and
$mn/(m+n)\ge1$, the value $c_0(f_n)=2(m+n)/(mn)$; halving to match the real
convention used throughout this paper recovers exactly
$\mathrm{rlct}_0(f)=(m+n)/(mn)$ above.
\end{remark}

Let $K=1/f$. The family is quasi-homogeneous with weights
$w=(n',m')=(n,m)/\gcd(m,n)$ and degree $d=mn/\gcd(m,n)$, so
Proposition~\ref{prop:quasihomog} yields the exceptional contribution
\[
\frac{|w|}{d+w_{\max}}
=\frac{m+n}{mn+\max(m,n)} .
\]
Moreover the real zero set of $f$ is $\{0\}$ precisely when $m$ and $n$ are
both even; otherwise it is a curve through the origin, whose smooth part
contributes the additional value $1/2$. Hence
$\nabla K\in L^p_{\mathrm{loc}}(\mathbb R^2)$ if and only if $p<p^*$, where
\[
p^*=
\begin{cases}
\dfrac{m+n}{mn+\max(m,n)} & \text{if $m,n$ even},\\[1em]
\min\!\Bigl(\dfrac{m+n}{mn+\max(m,n)},\dfrac12\Bigr) & \text{otherwise}.
\end{cases}
\]
In every case $p^*>\overline{p}$ strictly, in accordance with
Theorem~\ref{thm:admissibility}\,(iii): the divisors over the origin carry
$M_i\ge1$, so the RLCT systematically underestimates the exact threshold.

\medskip
\noindent\textbf{Representative cases.} Each exact value below is confirmed
by the anisotropic dyadic scaling of Proposition~\ref{prop:quasihomog} and,
where the real zero set is a curve, by the transversal estimate
$|\nabla K|\asymp d^{-2}$ near its smooth points.
\[
\begin{array}{c|c|c|c|c|c}
f(x,y)
& Z_{\mathbb R}(f)
& \mathrm{rlct}_0(f)
& \overline{p}=\tfrac12\,\mathrm{rlct}_0(f)
& \tfrac{m+n}{mn+\max(m,n)}
& p^*
\\ \hline
x^{2}+y^{3}
& \text{cusp curve}
& \dfrac{5}{6}
& \dfrac{5}{12}
& \dfrac{5}{9}
& \dfrac{1}{2}
\\[0.5em]
x^{4}+y^{6}
& \{0\}
& \dfrac{5}{12}
& \dfrac{5}{24}
& \dfrac{1}{3}
& \dfrac{1}{3}
\\[0.5em]
x^{3}+y^{5}
& \text{curve}
& \dfrac{8}{15}
& \dfrac{4}{15}
& \dfrac{2}{5}
& \dfrac{2}{5}
\end{array}
\]
The first line reproduces Example~\ref{ex:cusp}: the smooth part of the cusp
curve ($1/2$) dominates the exceptional divisor ($5/9$). In the third line
the situation is reversed: the exceptional contribution $2/5$ lies below the
curve contribution $1/2$ and determines the threshold. The second line has
an isolated real zero and the exceptional divisor alone decides. In all
three cases the RLCT bound $\overline{p}$ is strictly smaller than $p^*$.

Larger exponents $(m,n)$ correspond to stronger singularities of $f$ at the
origin and therefore, once the exceptional contribution dominates, to a
smaller admissible range of Morrey exponents for the velocity field
$\nabla K$.

\subsection{A mixed anisotropic example}
\label{subsec:mixed_example}

We next consider a genuinely mixed case,
\[
f(x,y)=x^2+xy^3+y^5,
\]
which illustrates how lower--order mixed terms may alter the anisotropy of level
sets without changing the dominant Newton geometry.

The Newton polyhedron has vertices at $(2,0)$, $(1,3)$, and $(0,5)$.
Note that $(1,3)$ lies strictly above the segment connecting $(2,0)$ and $(0,5)$,
meaning it does not belong to the boundary of the convex hull that intersects
the diagonal. The dominant face is thus the segment joining $(2,0)$ and $(0,5)$.
Intersecting this face with the diagonal $(t,t)$ yields
\[
\frac{t}{2}+\frac{t}{5}=1,
\qquad\text{so that}\qquad
t=\frac{10}{7} \approx 1.43.
\]
Hence
\[
d_{\mathrm{NP}}(f)=\frac{10}{7},
\qquad
\mathrm{rlct}_0(f)=\frac{7}{10}.
\]

\begin{figure}[h]
\centering
\begin{tikzpicture}[scale=1.0]
    \begin{scope}
        \clip (0,0) rectangle (5,6);
        \fill[blue!10] (0,5) -- (2,0) -- (5,0) -- (5,6) -- (0,6) -- cycle;
    \end{scope}

    \draw[->] (0,0) -- (5,0) node[right] {$i$};
    \draw[->] (0,0) -- (0,6) node[above] {$j$};

    \draw[dashed, gray] (0,0) -- (4.5,4.5) node[right, black] {$i=j$};

    \draw[thick, red] (0,6) -- (0,5) node[left, black] {$(0,5)$}
          -- (2,0) node[below, black] {$(2,0)$} -- (5,0);

    \fill (0,5) circle (2pt);
    \fill (2,0) circle (2pt);
    \fill (1,3) circle (2pt) node[above right] {$(1,3)$};

    \coordinate (D) at (1.428, 1.428);
    \fill[black] (D) circle (2pt);
    \draw[dotted] (D) -- (1.428,0) node[below] {$\frac{10}{7}$};
    \node[above right, scale=0.9] at (D) {$(d_{\mathrm{NP}}, d_{\mathrm{NP}})$};

    \node at (3.5,4.5) {$\NP(f)$};
\end{tikzpicture}
\caption{Newton polyhedron for $f(x,y)=x^2+xy^3+y^5$. The mixed term $(1,3)$ is redundant for the calculation of $d_{\mathrm{NP}}$.}
\end{figure}
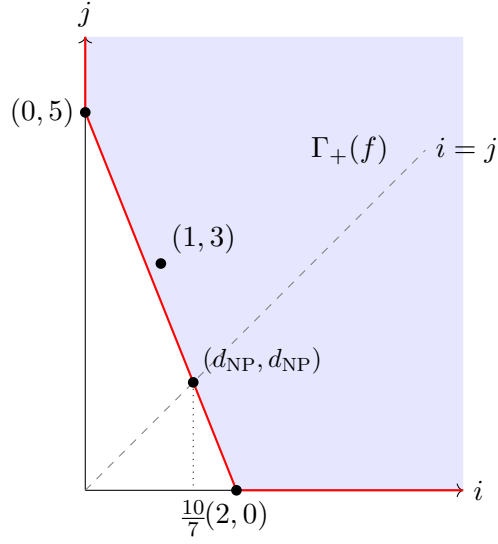

As in the previous family, the given coordinates are already adapted in the
sense of Theorem~\ref{thm:real-newton} (consistently with $d_{\mathrm{NP}}=10/7\ge1$),
so the RLCT bound gives the sufficient condition
$p<\overline{p}=\tfrac12\,\mathrm{rlct}_0(f)=\tfrac{7}{20}$.

The exact threshold is larger, and certifying it rigorously requires two
separate mechanisms, since the mixed monomial $xy^3$ leaves the dominant
exceptional contribution unchanged but must be treated separately near the
real zero curve.

\emph{Mechanism 1: away from the zero set of the principal part.} The
function is semi-quasi-homogeneous with principal part $x^2+y^5$ for the
weights $w=(5,2)$ and degree $d=10$: the mixed monomial $xy^3$ has weighted
degree $11>10$, so it is subordinate to the grading wherever the principal
part $x^2+y^5$ itself does not vanish to higher order, and there the
exceptional contribution of the weighted blow-up coincides with that of the
principal part alone, computed via Proposition~\ref{prop:quasihomog}. The
exceptional data are
\[
\nu_E=10,\qquad a_E=|w|-1=6,\qquad
M_E=d-w_{\max}=5,
\]
the latter realized by $\partial_x f=2x+y^3$, of weighted order
$5$; hence this mechanism gives the contribution
\[
\frac{a_E+1}{2\nu_E-M_E}=\frac{7}{20-5}=\frac{7}{15}.
\]

\emph{Mechanism 2: near the real zero curve $\{f=0\}$.} On the strict
transform of the real curve $\{f=0\}$ itself (solving the quadratic in $x$
shows real branches for $y\le0$, reduced and smooth away from the origin),
the principal part $x^2+y^5$ vanishes and the mixed term $xy^3$ is no
longer subordinate to it, so this contribution is \emph{not} read off the
weighted blow-up of the principal part; it is instead governed by the
transverse analysis of Lemma~\ref{lem:chain_rule} along the (reduced,
$\nu=1$, $M=0$) strict-transform component, giving the smooth-part ratio
$1/(\kappa+1)=1/2$.

Since $7/15<1/2$, Mechanism~1 realizes the divisorial minimum, and
\[
\nabla K\in L^p_{\mathrm{loc}}(\mathbb R^2)
\quad\Longleftrightarrow\quad
p<p^*=\min\Bigl(\tfrac{7}{15},\tfrac12\Bigr)=\frac{7}{15}.
\]

We caution that the estimate $|f|\sim\varepsilon^{10}$ under the weighted
scaling $x=\varepsilon^5X$, $y=\varepsilon^2Y$ cannot be asserted uniformly
on the full annulus $\{\varepsilon\lesssim|(x,y)|_w\lesssim2\varepsilon\}$:
on the zero set of the principal part, $X^2+Y^5=0$, the leading term
cancels and $f$ reduces to the subordinate monomial $\varepsilon^{11}XY^3$,
so $|f|$ is of strictly higher order than $\varepsilon^{10}$ near that
locus. The scaling estimate $|f|\sim\varepsilon^{10}$---and the resulting
contribution $7/15$---is uniformly valid only away from a shrinking
neighbourhood of $\{X^2+Y^5=0\}$; the complementary region near that locus
is exactly Mechanism~2, and since it produces the strictly larger ratio
$1/2$, it does not affect the minimum. It is this two-mechanism argument,
rather than a single uniform scaling estimate over the whole annulus, that
rigorously certifies $p^*=7/15$.

This example shows that mixed monomials can substantially modify the
anisotropy of the singularity while leaving the dominant Newton face
unchanged, and again exhibits the strict gap
$\overline{p}=\tfrac{7}{20}<\tfrac{7}{15}=p^*$ between the RLCT prediction
and the exact divisorial threshold.

\subsection{Isotropic and near--isotropic families}
\label{subsec:isotropic_families}

We briefly summarize several instructive families, emphasizing how the admissible
range interpolates between isotropic and strongly anisotropic regimes.

\begin{itemize}
\item \textbf{Isotropic case:} $f=x^a+y^a$, $a\ge2$.
The polyhedron is symmetric with respect to the diagonal, and the distance is
simply half the degree:
\[
d_{\mathrm{NP}}(f)=\frac{a}{2}, \qquad \mathrm{rlct}_0(f)=\frac{2}{a},
\qquad \overline{p}=\frac{1}{a}.
\]
The exact threshold follows from Proposition~\ref{prop:quasihomog} with
$w=(1,1)$, $d=a$, $w_{\max}=1$:
\[
p^*=\frac{2}{a+1},
\]
both for $a$ even (isolated real zero) and for odd $a\ge3$, where the line
$x=-y$ contributes $1/2\ge2/(a+1)$. The simplest instance $a=2$,
$f=x^2+y^2$, already exhibits the strict gap: $|\nabla K|\asymp r^{-3}$ in
polar coordinates gives $p^*=2/3$, while $\overline{p}=1/2$; the discrepancy
is due to $M_E=1\ge1$ on the exceptional divisor of the blow-up
(Theorem~\ref{thm:admissibility}\,(iii)).

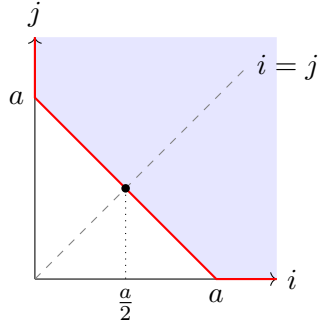
\begin{figure}[h!]
\centering
\begin{tikzpicture}[scale=0.8]
    \begin{scope}
        \clip (0,0) rectangle (4,4);
        \fill[blue!10] (0,3) -- (3,0) -- (4,0) -- (4,4) -- (0,4) -- cycle;
    \end{scope}
    \draw[->] (0,0) -- (4,0) node[right] {$i$};
    \draw[->] (0,0) -- (0,4) node[above] {$j$};
    \draw[dashed, gray] (0,0) -- (3.5,3.5) node[right, black] {$i=j$};
    \draw[thick, red] (0,4) -- (0,3) node[left, black] {$a$} -- (3,0) node[below, black] {$a$} -- (4,0);
    \coordinate (D) at (1.5,1.5);
    \fill (D) circle (2pt);
    \draw[dotted] (D) -- (1.5,0) node[below] {$\frac{a}{2}$};
\end{tikzpicture}
\caption{Isotropic case: the diagonal hits the midpoint of the face.}
\end{figure}

\item \textbf{Degenerate-polyhedron case:} $f=x+y^n$ (with $n$ large).
The face is very ``steep'', pulling the intersection point
$d_{\mathrm{NP}}=n/(n+1)$ below $1$---a signal that the coordinates are
\emph{not} adapted in the sense of Theorem~\ref{thm:real-newton}
(Remark~\ref{rem:adapted_coords}). Indeed, $\nabla f=(1,ny^{n-1})$ never
vanishes: $f$ is smooth, its zero set is a smooth curve, and
\[
\mathrm{rlct}_0(f)=1,
\qquad
p^*=\overline{p}=\frac12
\quad(\kappa=1),
\]
by Theorem~\ref{thm:admissibility}\,(iii); the naive substitution of
$d_{\mathrm{NP}}<1$ into the RLCT formula would give the incorrect value
$(n+1)/n>1$. This family is the equality case of the RLCT bound: the
threshold is governed entirely by the smooth reduced zero set
($\nu=1$, $M=0$, $a=0$), and no exceptional divisor intervenes.

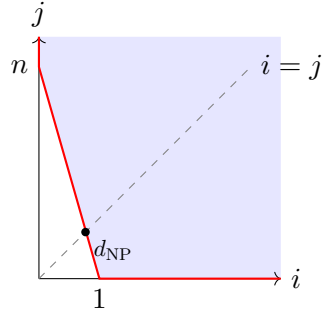
\begin{figure}[h!]
\centering
\begin{tikzpicture}[scale=0.8]
    \begin{scope}
        \clip (0,0) rectangle (4,4);
        \fill[blue!10] (0,3.5) -- (1,0) -- (4,0) -- (4,4) -- (0,4) -- cycle;
    \end{scope}
    \draw[->] (0,0) -- (4,0) node[right] {$i$};
    \draw[->] (0,0) -- (0,4) node[above] {$j$};
    \draw[dashed, gray] (0,0) -- (3.5,3.5) node[right, black] {$i=j$};
    \draw[thick, red] (0,4) -- (0,3.5) node[left, black] {$n$} -- (1,0) node[below, black] {$1$} -- (4,0);
    \coordinate (D) at (0.77,0.77);
    \fill (D) circle (2pt);
    \node[below right, scale=0.8] at (D) {$d_{\mathrm{NP}}$};
\end{tikzpicture}
\caption{Degenerate-polyhedron case $f=x+y^n$: the diagonal meets $\NP(f)$ at $d_{\mathrm{NP}}=n/(n+1)<1$, indicating non-adapted coordinates; the germ is in fact smooth.}
\end{figure}

\item \textbf{Near--unidirectional case:} $f=x^m+y$.
This is the dual of the previous case, where the polyhedron is stretched along
the $i$-axis; here $d_{\mathrm{NP}}=m/(m+1)<1$, again signalling non-adapted
coordinates. As before, $\nabla f=(mx^{m-1},1)$ never vanishes, $f$ is
smooth, and
\[
\mathrm{rlct}_0(f)=1,
\qquad
p^*=\overline{p}=\frac12
\quad(\kappa=1).
\]

\begin{figure}[h!]
\centering
\begin{tikzpicture}[scale=0.8]
    \begin{scope}
        \clip (0,0) rectangle (4,4);
        \fill[blue!10] (0,1) -- (3.5,0) -- (4,0) -- (4,4) -- (0,4) -- cycle;
    \end{scope}
    \draw[->] (0,0) -- (4,0) node[right] {$i$};
    \draw[->] (0,0) -- (0,4) node[above] {$j$};
    \draw[dashed, gray] (0,0) -- (3.5,3.5) node[right, black] {$i=j$};
    \draw[thick, red] (0,4) -- (0,1) node[left, black] {$1$} -- (3.5,0) node[below, black] {$m$} -- (4,0);
    \coordinate (D) at (0.77,0.77);
    \fill (D) circle (2pt);
    \node[above left, scale=0.8] at (D) {$d_{\mathrm{NP}}$};
\end{tikzpicture}
\caption{Near-unidirectional case $f=x^m+y$: the polyhedron flattens against the axis, and again $d_{\mathrm{NP}}<1$ reflects non-adapted coordinates for a smooth germ.}
\end{figure}
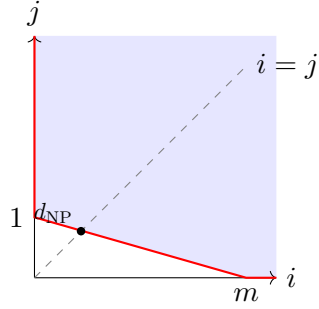
\end{itemize}

These families illustrate how Morrey admissibility varies with the
degree of anisotropy encoded in the Newton polyhedron in the genuinely
singular regime, and how the framework degenerates gracefully to the value
$1/(\kappa+1)$---governed by the smooth part of the zero set---as the
singularity disappears.
\subsection{Comparison with the classical Newtonian kernel}
\label{subsec:newtonian}

As a benchmark, consider the classical Newtonian interaction kernel
\[
K_{\mathrm{Newt}}(x)\sim |x|^{2-d}, \qquad d\ge3,
\]
for which
\[
|\nabla K_{\mathrm{Newt}}(x)|\sim |x|^{1-d}.
\]
It is well known that
\[
\nabla K_{\mathrm{Newt}}\in L^p_{\mathrm{loc}}
\quad\Longleftrightarrow\quad
p<\frac{d}{d-1}.
\]

This classical threshold is recovered \emph{exactly} by
Theorem~\ref{thm:admissibility}. Indeed,
$K_{\mathrm{Newt}}=|f|^{-\kappa}$ for
\[
f(x)=|x|^2,
\qquad
\kappa=\frac{d-2}{2},
\]
which satisfies Hypothesis~\ref{hyp:dominant} with this value of $\kappa$:
$|\nabla f|\,|f|^{-(\kappa+1)}=2|x|\cdot|x|^{-d}=2|x|^{1-d}$. The blow-up
of the origin in $\mathbb R^d$ log-resolves $f$ and its Jacobian ideal
$\mathcal J_f=(x_1,\dots,x_d)$, with a single exceptional divisor $E$
carrying
\[
\nu_E=2,\qquad a_E=d-1,\qquad M_E=1 .
\]
Hence
\[
p^*(\kappa)=\frac{a_E+1}{(\kappa+1)\nu_E-M_E}
=\frac{d}{2\kappa+1}
=\frac{d}{d-1},
\]
in perfect agreement with the classical computation. Note the essential role
of the corrected gradient order: $M_E=1$ (the Jacobian ideal is the maximal
ideal, of order one along $E$), consistent with
$0\le M_E\le\nu_E-1=1$ from Lemma~\ref{lem:chain_rule}. The RLCT here is
$\mathrm{rlct}_0(|x|^2)=(a_E+1)/\nu_E=d/2$ (the real zero set is the origin
alone, so no strict-transform component caps the threshold at $1$;
cf.\ Remark~\ref{rem:rlct_divisorial}), and, since $\kappa+1=d/2$, the RLCT
bound reads
\[
\overline{p}(\kappa)
=\frac{\mathrm{rlct}_0(f)}{\kappa+1}
=\frac{d/2}{d/2}
=1
\;<\;\frac{d}{d-1}
=p^*(\kappa),
\]
strictly smaller than the exact threshold, as always in the isolated-zero
regime.

The strength of the resolution--based approach thus lies both in this exact
consistency with the classical isotropic theory and in its robustness under
anisotropic and heterogeneous perturbations: in the latter settings, where
no radial structure is available, Newton geometry and divisorial data
provide an immediate and systematic identification of the admissible
singularity regime without requiring case-by-case analysis.

\section*{Conclusion}

This work establishes a complete geometric characterisation of the exact
local integrability threshold of $\nabla K$ for singular interaction
kernels: under the two-sided comparability assumption of
Hypothesis~\ref{hyp:dominant}, $\nabla K\in
L^p_{\mathrm{loc}}$ if and only if
$p < p^*(\kappa):=\min_i (a_i+1)/\bigl((\kappa+1)\nu_i - M_i\bigr)$, where the
divisorial data $(\nu_i, M_i, a_i)$ arise from a real log-resolution of $f$
together with its Jacobian ideal, and the minimum runs over exceptional
divisors and strict-transform components alike. The real log-canonical
threshold of $f$ provides the computable lower bound
$\overline p(\kappa):=\mathrm{rlct}_0(f)/(\kappa+1)$ for this exact
threshold, with equality $p^*(\kappa)=\overline p(\kappa)$ if and only if
some divisor computing $\mathrm{rlct}_0(f)$ has $M_i=0$---necessarily a
reduced, generically smooth component of the strict transform of the real
zero set---and strict inequality throughout the genuinely singular
regime---in particular for every isolated zero in dimension at least two,
including the (quasi-homogeneous and semi-quasi-homogeneous) examples
treated in Section~3. For
quasi-homogeneous singularities the exact threshold takes the closed form
$|w|/(\kappa d+w_{\max})$, possibly capped by the smooth-part value
$1/(\kappa+1)$, and the classical Newtonian kernel is recovered exactly as
the isotropic instance $f=|x|^2$, $\kappa=(d-2)/2$.

We have further made explicit, in Section~\ref{sec:morrey_bridge}, the
precise bridge between this purely geometric $L^p_{\mathrm{loc}}$ statement
and the well-posedness theorem of \cite[Thm.~3.1]{suleiman2023existence}.
As stated there, the only requirement on the
kernel is the \emph{global} condition $\nabla K\in L^1(\mathbb R^n)$,
entirely independent of the Morrey scaling parameter $\lambda$ of the
solution space; we showed that $p^*(\kappa)>1$ is the exact local
condition for the singular part of $K$ to contribute $\nabla K\in
L^1_{\mathrm{loc}}$, and that this local condition, together with a mild
global truncation of $K$ away from its singularity
(Corollary~\ref{cor:morrey_bridge}), yields the global hypothesis
$\nabla K\in L^1(\mathbb R^n)$ needed for
\cite[Thm.~3.1]{suleiman2023existence}, with the computable sufficient
condition
$\kappa+1<\mathrm{rlct}_0(f)$. We verified this bridge explicitly for the
Newtonian kernel, where the RLCT bound is exactly borderline while the
exact threshold certifies \cite[Thm.~3.1]{suleiman2023existence} for every
$d\ge3$ and, once certified, simultaneously for the \emph{entire} admissible
range of $\lambda$---and we noted (Remark~\ref{rem:regime_C_examples}) that
several of the anisotropic examples of Section~\ref{sec:examples} only
enter this regime for sufficiently mild exponents $\kappa$, again
independently of $\lambda$. The full nonlinear well-posedness conclusion for
the aggregation equation~\eqref{eq:agg} then follows by invoking
\cite[Thm.~3.1]{suleiman2023existence} as a black box, which we do not
reprove here; the contribution of the present paper is the geometric input
to that step, made precise in Corollary~\ref{cor:morrey_bridge}.

The exact $L^p_{\mathrm{loc}}$ threshold is thus resolution-theoretic in
nature: it is controlled not by the RLCT alone but by the divisorial geometry
of the pair $(f,\mathcal J_f)$, with the gradient orders $M_i$ measuring
precisely how far the sharp threshold exceeds the RLCT prediction. The RLCT
remains a robust, computationally accessible proxy that captures the leading
concentration mechanism near the singular set. This perspective clarifies
why admissible regimes are stable under anisotropic perturbations, and why
isotropic and genuinely anisotropic kernels fit within a common geometric
framework.

The resolution-based approach developed here suggests several directions for
further study. One concerns aggregation dynamics on singular or stratified
spaces, where resolution invariants may encode the geometry of concentration
sets. Another concerns gradient-flow models in optimal transport
\cite{figalli2010new}, where anisotropic scaling properties may be analyzed
through Newton polyhedra. Related questions also arise in flocking and alignment
models with directionally dependent interactions \cite{shvydkoy2021global}.

From an algebraic perspective, valuative methods \cite{jonsson2012valuations}
may help refine admissibility criteria, while the ACC property for log-canonical
thresholds \cite{hacon2014acc} suggests the possibility of uniform geometric
bounds in families of models. The connection with multiplier ideals
\cite{blickle2004informal} raises a related question on the analytic side,
namely whether analogous functional structures can be associated with Morrey
admissibility classes.

Beyond the examples considered here, it would be natural to extend this
framework to matrix-valued kernels, measure-driven dynamics, and more general
singular configurations. Kernels whose singular behavior is not modeled by a
single analytic divisor remain open, and will likely require new analytical
tools together with a more refined interaction between resolution data and
stratified geometry.


\subsection*{Acknowledgment}
The first author would like to thank the S\~ao Paulo Research Foundation
(FAPESP) for support under grant 2019/21181-0 and Coordena\c{c}\~ao de
Aperfei\c{c}oamento de Pessoal de N\'ivel Superior -- Brasil (CAPES) for
support through the MATH-AmSud program, Grant No.\ 88881.179491/2025-01.

\subsection*{Author contributions}
Both authors contributed equally to this manuscript.

\subsection*{Data availability}
No datasets were generated or analyzed during the current study.

\subsection*{Declarations}
\textbf{Competing interests.}
The authors declare that they have no competing interests.

\end{document}